\documentclass[smallextended]{svjour3}

\usepackage{amssymb,amsfonts, amsmath} % 
\usepackage{bm}
\usepackage{todonotes}
\usepackage{mathtools, relsize,geometry}
\usepackage{mathrsfs}
\usepackage{multirow,url}
\usepackage{soul}

\DeclareMathOperator*{\argmin}{arg\,min}

\makeatletter

\renewcommand{\vec}[1]{%
	\ifcat\relax\noexpand#1%
	\ensuremath{\boldsymbol{\lowercase{#1}}}%
	\else
	\ensuremath{\mathbf{\lowercase{#1}}}%
	\fi
}
\newcommand{\R}{\ensuremath{\mathbb{R}}}

\newcommand{\T}{\top\!}
\newcommand{\SF}[1]{\left\|#1\right\|_{S,F}}
\newcommand{\Sdue}[1]{\left\|#1\right\|_{S,2}}

\usepackage{etoolbox}

\usepackage{pgfplots}
\usepackage{graphicx,epstopdf} 
\usepackage{subfig}
\usepackage{tikz}
\usepackage{algorithm}
\usepackage{algpseudocode}
\usepackage{algorithmicx}
\usepackage[normalem]{ulem}
\usepackage{cancel}

\usetikzlibrary{shapes.geometric}

\definecolor{matlabred}{rgb}{0.9047,    0.1918,    0.1988}
\definecolor{matlabblue}{rgb}{0.2941    0.5447    0.7494}
\definecolor{matlabgreen}{rgb}{	0.3718    0.7176    0.3612}
\definecolor{matlaborange}{rgb}{1.0000    0.5482    0.1000}

\usepackage{algpseudocode}
\algnewcommand{\LineComment}[1]{\Statex \(\%\) \small \textit{#1} \(\%\)}

\usepackage{algorithm}
\title{$S^{\top\!}S$-SVD {\color{black}via sketching} and the nearest {\color{black}$S^{\top\!}S$-}orthogonal matrix%
\thanks{Version of \today}}
\author{Davide Palitta\thanks{Dipartimento di Matematica and (AM)$^2$,
Alma Mater Studiorum Universit\`a di Bologna,
Piazza di Porta San Donato  5, I-40127 Bologna, Italy,
{\tt \{davide.palitta, valeria.simoncini\}@unibo.it}}
 \and Valeria Simoncini\thanks{IMATI-CNR, Pavia, Italy. }}

\begin{document}
\maketitle

%\renewcommand{\thefootnote}{\fnsymbol{footnote}}
%\maketitle \pagestyle{myheadings} \thispagestyle{plain}
%\markboth{ D.\ PALITTA and V.\ SIMONCINI}{\sc $S^\T S$-svd and the nearest sketched orthogonal matrix}

%% ------------------------------------------------------------------
%% ABSTRACT
%% ------------------------------------------------------------------

\vskip 0.1in
\begin{center}
{\it This paper is dedicated to \AA{}ke Bj\"orck, on the occasion
of his 90th birthday.}
\end{center}
\vskip 0.1in

\begin{abstract}
Sketching techniques have gained popularity in numerical linear algebra 
to accelerate the solution of least squares problems. 
The so-called $\varepsilon$-subspace embedding property of a
sketching matrix $S$ has been largely used to characterize 
the problem residual norm, since the procedure is no longer optimal 
in terms of the (classical) Frobenius or Euclidean norm. 

By building on available results on the SVD of the sketched matrix $SA$ derived by
Gilbert, Park, and Wakin (Proc. of SPARS-2013),
 a novel decomposition of $A$,
the $S^\T S$-SVD, is proposed, which {\color{black}\emph{holds} with high probability,} and in which
the left singular vectors  are orthonormal
with respect to a (semi-)norm defined by the sketching matrix $S$. 
The new decomposition is less expensive to compute {\color{black}than the standard SVD}, while preserving
the singular values with probabilistic confidence.
The $S^\T S$-SVD appears to be the right tool to analyze the quality of
several {\color{black}sketching-based} techniques in the literature, for which examples are reported.
For instance, it is possible to simply bound the distance from (standard) orthogonality 
of {\color{black}sketching-based} orthogonal matrices in state-of-the-art {\color{black}randomized algorithms for QR factorizations}.
As an application, the classical problem of the nearest orthogonal matrix is generalized
 to the new $S^\T S$-orthogonality, and 
 the $S^\T S$-SVD is used to solve it. Probabilistic bounds on the quality of the solution
are also derived.
\end{abstract}

%\begin{keywords}
%Sketching, SVD, nearest orthogonal matrix.
%\end{keywords}

%\begin{MSCcodes}
%\end{MSCcodes}

%% ------------------------------------------------------------------
%% END HEADER
%% ------------------------------------------------------------------

%%%%%%%%%%%%%%%%%%%%%%%%%%%%%%%%%%%%%
\section{Introduction}

The inspiration for this work
 comes from~\cite{Bjorck} where, for a given
matrix $A\in{\mathbb R}^{m\times n}$, Bj\"orck and Bowie studied the problem
\begin{equation}\label{eq:main0}
 \min_{\substack{Q\in\mathbb{R}^{m\times n} \\Q^\T Q=I}}\|A-Q\|_{*},
\end{equation}
with $m\geq n$ and $*=2,F$, that is, the spectral and Frobenius norms.
The solution $Q$ to~\eqref{eq:main0} is given by the 
orthogonal factor of the polar decomposition (\cite[Th.1.2.4]{Bjorck.book.96})
of $A$; see \cite{Bjorck}, and, e.g., the presentation and bibliography
in \cite{doi:10.1137/090765018}. 
However, this decomposition may be expensive to compute for large-scale problems,
hence Bj\"orck and Bowie proposed an iterative algorithm to solve~\eqref{eq:main0}; see~\cite[Section 2]{Bjorck}.

With the same desire of accelerating the solution of~\eqref{eq:main0} 
{\color{black} when $A$ has large rank},  we instead explore the use of
state-of-the-art randomization-based strategies known as sketchings.
Since the seminal work~\cite{Sarlos2006}, sketching techniques have shown 
their potential in speeding up the numerical solution of {\color{black}massively overdetermined}
least squares problems; see, e.g.,~\cite{Blendenpik,RokTyg2008}.
%Roughly speaking, 
These strategies reduce the row dimension of the coefficient matrix 
by projection, using suitable linear maps.
Under certain probabilistic conditions, these maps can be ensured to have good
metric properties. More precisely,
%Thanks to their appealing properties, oftentimes these transformations amount to 
%oblivious $\varepsilon$-subspace embeddings. In particular, 
for $\varepsilon\in (0,1)$ and
for any vector $v$ in a $k$-dimensional vector space $\mathcal{V}\subset\mathbb{R}^m$, it is possible to construct a linear map, called
an {\it oblivious $\varepsilon$-subspace embedding} 
$S:\mathbb{R}^m\mapsto\mathbb{R}^s$, $s\ll m$, such that 
$|\|Sv\|^2-\|v\|^2|\leq\varepsilon$ with high probability.
We remark that $S$ can be selected without knowing 
$\mathcal{V}$ itself but relying only on its dimension $k$.
Applying sketching to a least squares problem can be interpreted
 as recasting the original problem in terms of a different norm, 
the $S^\T S$-norm $\|v\|_{S^\T S}^2=v^\T S^\T Sv$. 
Although $S^\T S$ is only semidefinite in general, the above metric property
ensures that with high probability $S^\T S$ 
 defines a positive definite norm 
on the embedded space; see, e.g.,~\cite{BalabanovNouy19}.

We aim to explore this non-standard norm and its theoretical and
 computational properties
in the solution of \eqref{eq:main0}. 
%is rather straightforward, the analysis 
%of the obtained sketched problem is not as easy. 
To this end, by building upon available results on the SVD of $SA$
\cite{Gilbertetal.14}, we propose a novel decomposition 
of $A$ that we name $S^\T S$-SVD, which {\color{black}\emph{holds}}
{with high
probability}. This amounts to an SVD-like decomposition 
of $A$, where the left singular vectors $w_j$s are $S^\T S$-orthonormal,
namely $w_j^\T S^\T Sw_i=\delta_{ij}$, $i,j=1,\ldots,n$; see section~\ref{SS-SVD}.
The magnitude of the $S^\T S$-singular values of $A$ can be related to their standard counterparts via the threshold $\varepsilon$ associated with the adopted embedding 
$S$~\cite{Gilbertetal.14}. 
This is one of the key features of the $S^\T S$-SVD that allow us to
analyze the sketched version of~\eqref{eq:main0}. Moreover, we can 
derive explicit relations between the solution to the problem in the non-standard
norm and the solution $Q$ to~\eqref{eq:main0}.

{\color{black}
The idea of sketching a matrix's SVD is certainly not new, and it has been introduced together with the development of sketching algorithms \cite{Halko2010}, \cite{Martinsson_Tropp_2020},\cite{Sarlos2006}. However, to our knowledge, {\it sketched SVD}  usually refers to the construction of a reduced surrogate of $A$, built %using an $S^\top S$-orthogonality 
to preserve with high probability some of the properties of the singular values of $A$ and of the range of $A^T$ \cite{Gilbertetal.14},\cite{RokTyg2008}. Alternatives include randomized SVD associated with range-finding strategies; see \cite{Kireeva.Tropp.tr23} and the references therein. This work} {\color{black} %The $S^\T S$-SVD
provides a new point of view on the use of SVD-related sketching techniques, as it aims at constructing a full decomposition of $A$, by computing an  $S^\top S$-orthogonal basis for the range space. 
Indeed, the %latter ones are traditionally employed to compute approximations, or \emph{sketches}, of the object of interest, which is tipically too large to handle. The underlying assumption is that the computed sketches preserve, to some extent, some of the quantities of interest of the original object and the $\varepsilon$-subspace embedding property is often employed to characterize such preservation. The 
$S^\T S$-SVD computes a factorization of  $A$, holding
with high probability.}
%{\color{red}, where the exactness of such factorization is subject to a failure probability due to the randomized nature of the sketching $S$. In principle, there is no reason to prefer the standard SVD over the $S^\T S$-SVD when the latter is successful: they both provide important insights on the factorized matrix $A$ like its rank or a basis of its range with ($S^\T S$-)orthonormal columns; see section~\ref{SS-SVD}. }
{\color{black}By relying on the concept of $S^TS$-orthogonality, we are able to leverage the computational gains of sketching while attaining full factorizations with probabilistic confidence. }

We believe that the $S^\T S$-SVD will be a crucial tool also in attaining 
a complete understanding of sketching techniques applied to general least 
squares problems. For instance, several results available in the literature are
natural consequences of the $S^\T S$-SVD. 
To evaluate the quality of the new orthogonality constraint, we also
estimate the distance from Euclidean orthogonality of the computed 
$S^\T S$-orthogonal matrices.   
%We also derive novel results bounding quantities of the form 
%$\|P^\T P-I\|_*$, $*=2,F$, where $P\in\mathbb{R}^{m\times n}$ has 
%$S^\T S$-orthonormal columns. 
These can be used, for instance, to provide 
certain orthogonality
guarantees in state-of-the-art {\color{black}randomized algorithms for QR factorizations}.

 As possible motivating applications, we 
envision the employment of the $S^\T S$-SVD also in other settings. 
As an example, orthogonal matrices are a fundamental 
ingredient in problems %with orthogonality constraints,
where the solution is constrained to belong to a space of matrices having orthonormal
columns, the Stiefel manifold; see, e.g., \cite{Absiletal.08},\cite{eas:99}. {\color{black} These types
of constraints are} particularly convenient in the solution of certain
differential equations, because the orthogonal space allows the solution method
to preserve key properties of the dynamical system \cite{Celledoni2002},\cite{Hairer2002}.
Beyond this, orthogonality may be a component of computational strategies enforcing a
low-rank manifold representation of the flow
  as time integration proceeds, see, e.g., \cite{Cerutietal.22}.
The use of sketched orthogonality, that is of a ``sketched'' Stiefel manifold,
may help decrease the computational costs of these procedures when
large systems of differential equations arise. In other words, this strategy
may be viewed as an $\varepsilon$-controlled quasi-orthogonality structure,
and may lead to the supervised relaxation of orthogonal-manifold-based models. 
%The $S^\T S$-SVD will be a remarkable ally in the study of this intriguing approach.

Next is a synopsis of the paper. 
In section~\ref{Subspace embeddings and randomized QR} we recall basic properties 
of oblivious $\varepsilon$-subspace embeddings and their use in the so-called 
randomized QR factorization. Section~\ref{SS-SVD} reports
 the derivation of the $S^\T S$-SVD, its properties along with a 
fast algorithm for its computation. Some important scenarios where the $S^\T S$-SVD can be adopted are explored in section~\ref{On the applicability} whereas in section~\ref{sec:distorth} we employ
the $S^\T S$-SVD to bound  the distance 
of $S^\T S$-orthogonal matrices from ``standard'' orthogonality. 
%, like the ones computed by randomized QR algorithms. 
The sketched version of~\eqref{eq:main0} is presented and analyzed 
in section~\ref{The nearest sketched orthogonal matrix}. The paper 
ends with some conclusions in section~\ref{Conclusions}.

All the experiments reported in this paper have been run using Matlab (version 2024b) on a machine
with a 1.2GHz Intel quad-core i7 processor with 16GB RAM on an Ubuntu 20.04.2
LTS operating system.

%{\color{red}
%Th problem~\eqref{eq:main0} is strictly related to the so-called (balanced) Procrustes problem. In particular, given $C$,
%$B\in\mathbb{R}^{m\times n}$, this problem can be expressed as follows
%%
%\begin{equation}\label{eq:Procrustes_def}
% \min_{\substack{Q\in\mathbb{R}^{n\times n} \\Q^\T Q=I}}\|CQ-B\|_{*};
%\end{equation}
%%
%see, e.g.,~\cite{Schoenemann_1966}.

%The minimization problem~\eqref{eq:Procrustes_def} finds applications in diverse settings: psychometric and factor analysis~\cite{Green_1952,Hurley_Cattell_1962,Meredith_1977}, global positioning system problems~\cite{Bell_2003}, and
%the improvement of orthogonal factors in the Tucker decomposition of tensors~\cite{HOOI}, to name a few.

%It has been shown that~\eqref{eq:Procrustes_def} is closely related to~\eqref{eq:main0}. Indeed, it holds
%%
%$$ \min_{\substack{Q\in\mathbb{R}^{n\times n} \\Q^\T Q=I}}\|CQ-B\|_{*}= \min_{\substack{Q\in\mathbb{R}^{n\times n} \\Q^\T Q=I}}\|Q-C^\T B\|_{*},$$
%%
%so that the solution $Q$ to~\eqref{eq:Procrustes_def} is given by the orthogonal factor of the polar decomposition of $C^\T B$.
%}

%Let $A=W R$ be the reduced QR decomposition of $A$.
%The solution to (\ref{eq:main0}) is obtained as $Q=W UV^\T$, where
%$U, V$ are the orthogonal matrices containing the
%left and right singular vectors
% of $R$, that is, $R=U\Sigma V^\T$ in the SVD of $R$.

\vskip 0.1in

{\it Notation.} We use $x^\T$ to denote the transpose of the vector $x$.
We use the Euclidean norm for vectors, and the associated induced matrix
norm for matrices, which we call 2-norm $\|\cdot\|_2$. We also report results using
the Frobenius matrix norm, denoted by $\|\cdot\|_F$.
Given a % full rank\todo{DP: Do we need ``full rank'' here? $\Theta$ can be rank deficient.} rectangular 
matrix $A$, we denote with $A^\dagger$ its pseudo-inverse
(Moore-Penrose inverse).
Exact arithmetic is assumed throughout.

With some abuse of notation, throughout the paper
we shall adopt the terminology
``orthogonal matrix'' also for tall rectangular matrices whose columns
are orthonormal, without requiring the matrix to be square.

%%%%%%%%%%%%%%%%%%%%%%%%%%%%%%%%%%%%%%
 \section{Subspace embeddings and randomized QR}\label{Subspace embeddings and randomized QR}
In the last decade, randomization-based tools have been shown to be an important 
aid in decreasing the computational cost of a number of algorithms in numerical 
linear algebra and scientific computing in general.

One such tool is given by (oblivious) $\varepsilon$-subspace 
embeddings; see, e.g.,~\cite[Section 8.7]{Martinsson_Tropp_2020}. Here
we use the notation adopted in \cite{RandomGS}.

\begin{definition}[{\cite[Definition 2.3]{RandomGS}}]
 Let $\mathcal{V}$ be a $k$-dimensional subspace of $\mathbb{R}^m$. Then, given $\varepsilon\in(0,1)$ and $s\leq m$, a linear map $S\in\mathbb{R}^{s\times m}$ is said to be an $(\varepsilon,\delta,k)$-subspace embedding for $\mathcal{V}$ if
\begin{equation}\label{eq:def_subspace_emb}
(1-\varepsilon)\|v\|^2 \le \|S v\|^2 \le (1+\varepsilon)\|v\|^2,\quad \text{for any }v\in\mathcal{V},
\end{equation}
holds with probability $1-\delta$, at least.
\end{definition}

We observe that from a deterministic point of view, $S^\T S$ defines a {\it semi}-norm,
whereas within a probabilistic setting, the definition above ensures that
$\|Sv\|^2$ is greater than zero for $v \in {\cal V}$, $v\neq 0$, with high probability;
%For our purposes, it is important to stress that an 
%$\varepsilon$-subspace embedding 
%induces the quadratic form $\|v\|_{S^\T S}^2:=v^\T S^\T S v$ which,
%with high probability, is positive definite on the space 
%that embeds; 
see, e.g.,~\cite[Proposition~3.3]{BalabanovNouy19}.
A fact following from (\ref{eq:def_subspace_emb}) is that with high probability it
holds that
\begin{equation}\label{eq:normS}
-\varepsilon\|v\|^2 \le v^\T (I -  S^\T S) v \le \varepsilon\|v\|^2,
\end{equation}
%so that in particular, $\|I-S^\T S\| \le \varepsilon$.
which measures the distance of $S^\T S$ from acting as the identity.

Various choices of randomized linear maps $S$ have been proposed in the literature. 
An incomplete list includes Gaussian transformations, sparse sign matrices, and subsampled trigonometric functions; see, e.g.,~\cite[Section 9]{Martinsson_Tropp_2020}. Given the 
dimension $k$ of the subspace to be embedded, the threshold $\varepsilon$, and the failure probability $\delta$, one can select suitable sketching dimensions $s$ to ensure~\eqref{eq:def_subspace_emb}.
Our derivations do not depend on the nature of the
%In this paper we are not interested in the nature of the 
sketching $S$ as long as~\eqref{eq:def_subspace_emb} holds. 
  Hence, in the following we will only assume
that the adopted $S$ is a {\color{black}randomized $(\varepsilon,\delta,\text{rank}(A))$}-subspace embedding
for $\text{Range}(A)$ so that the $S^\T S$-norm is well-defined on 
this space with high probability.

\begin{definition}\label{Def:SSnorms}
 Given a matrix $P\in\mathbb{R}^{m\times n}$ and a sketching matrix
$S\in\mathbb{R}^{s\times m}$, if $S$ is a {\color{black}randomized $(\varepsilon,\delta,\text{rank}(P))$}-subspace
embedding for $\text{Range}(P)$ then we can define the following matrix norms
$$
\SF{P}^2:={\rm trace}(P^\T S^\T S P),
$$
and
$$
\Sdue{P}^2:=\max_{\|x\|_2=1}x^\T P^\T S^\T S P x.
$$
\end{definition}
%\todo{DP: any other notation for the norms is welcome. Just change the macros}
Thanks to~\eqref{eq:def_subspace_emb}, it follows that % is easy to show that
$$
(1-\varepsilon)\|P\|_F^2 \le \SF{P}^2 \le (1+\varepsilon)\|P\|_F^2,\quad\text{and}\quad (1-\varepsilon)\|P\|_2^2 \le \Sdue{P}^2 \le (1+\varepsilon)\|P\|_2^2;
$$
see, e.g.,~\cite[Corollary 3.1]{Szyldetal2024}.
Note that both results also readily follow from using the singular 
value bounds in Theorem~\ref{Th:boundsingularvalues} below, 
derived in~\cite{Gilbertetal.14}.

Thanks to their ability in preserving norms up to a small distortion parameter, oblivious $\varepsilon$-subspace embeddings are the backbone of randomized algorithms {\color{black}for QR factorizations};
see, e.g.,~\cite{RandomGS,Szyldetal2024,RandGSreorth,RandGSHouse}. These procedures are able to cut down the cost of computing QR factorizations. 
On the other hand, they provide factorizations of the form $A=QR$ 
with \emph{almost} orthogonal $Q$ factor; we refer to 
section~\ref{sec:distorth} for a more detailed discussion of {\color{black} this distance to orthogonality}.

%\todo{VS:Is this correct?}
%%
%$$
%\SF{Q^\T Q-I}=0=\Sdue{Q^\T Q-I} .
%$$
%%%
%{\color{red}In Algorithm~\ref{alg:randGS} we report a naive implementation of a randomized Gram- Schmidt QR procedure; 
%see, e.g.,~\cite[Algorithm 2]{RandomGS}; 
%See also~\cite{RandGSreorth,RandGSHouse} for more robust orthogonalization variants.}
% and~\cite{Szyldetal2024} for the use of the multisketching paradigm in this context.
%
In the following we shall say that $Q$ is $S^\T S$-orthogonal, with high probability, if
$Q^\T S^\T S Q = I$, while $Q$ is orthogonal if $Q^\T Q=I$, where $I$ is the identity
matrix
of dimension equal to the number of columns of $Q$. 
%We note that different conditions associated with the term ``$S^\T S$-orthogonality''
% are sometime used in other contexts, such as, for instance,
%$Q^\T S^\T S Q = S^\T S$ \cite{Addref}.

%\begin{algorithm}[t!]
%\begin{algorithmic}[1]
%%\setstretch{1.2}
%\smallskip
%\Statex \textbf{Input:} $A=[a_1,\ldots,a_n]\in\mathbb{R}^{m\times n}$, $S\in\mathbb{R}^{s\times m}$.
%\Statex \textbf{Output:} $Q=[q_1,\ldots,q_n]\in\mathbb{R}^{m\times n}$, $R\in\mathbb{R}^{n\times n}$ such 
%that $A=QR$, $\SF{Q^\T Q-I}=\Sdue{Q^\T Q-I}=0$, $R$ upper triangular.
%\smallskip
%
%\State Set $s=Sa_1$, $R_{1,1}=\|s\|_2$, $q_1=a_1/R_{1,1}$, $s_1=s/R_{1,1}$
%\For{$k=2,\ldots,n$}
%\State Set $s=Sa_k$ and $q=a_k$
%\For{$j=1,\ldots,{k-1}$}
%\State  $R_{j,k}=s^\T s_j$\label{algline_innerproduct}
%\State  $s=s-R_{j,k}s_j$
%\State  $q=q-R_{j,k}q_j$
%\EndFor
%\State $R_{k,k}=\|s\|_2$
%\State $q_k=q/R_{k,k}$, $s_k=s/R_{k,k}$
%\EndFor
%\end{algorithmic}    \caption{Randomized Gram-Schmidt QR\label{alg:randGS} }
%\end{algorithm}

{One of the main advantages in employing the sketching matrix $S$ 
%in Algorithm~\ref{alg:randGS} 
is the decrease in the cost of the inner products,
% in line~\ref{algline_innerproduct} 
as vectors of length $s$ -- instead of $m$ -- are involved. }
On the other hand, $S$ needs to be applied to each of the 
$n$ columns of $A$. Therefore, matrix-vector products with $S$ 
must be as cheap as possible to obtain an
effective randomized orthogonalization;
see, e.g., the detailed analysis in \cite[Section 2.4]{RandomGS}. For instance, when $S$ amounts to a subsampled trigonometric transformation, 
like in Example~\ref{ex:1.1}--~\ref{ex:ex2} below, 
performing $Sx$ costs $\mathcal{O}(m\log s)$
 floating point operations (flops) if the action of $S$ is
judiciously implemented~\cite{WOOLFEetal2008}. This means that only $\mathcal{O}(mn\log s)$ flops are needed to compute $SA$. Notice that these costs can be further reduced in
case of a sparse $A$.

% for a more detailed analysis of 
%the computational performance of Algorithm~\ref{alg:randGS}.

Most available analyses on sketched QR {\color{black}procedures} focus on backward stability
issues, by  estimating the residual norm $\|A-QR\|_*$, and controlling the growth of
the condition number of $Q$, 
 $\kappa_*(Q)=\|Q\|_*\|Q^\dagger\|_*$,  for $*=2,F$; 
%and the overall stability of this numerical scheme; 
see, e.g.,~\cite{RandomGS,RandGSHouse}.
We enrich this analysis by investigating how far the columns of an $S^\T S$-orthogonal
matrix are from an orthogonal matrix.
If $S^\T S$ were positive definite this investigation would be very short, as
this distance would depend on the condition number of $S^\T S$. Our probabilistic
setting requires extra work.  
To give a glimpse of the type of expected results,
we report an upper bound on the cosine of the angle between two 
different columns of 
an $S^\T S$-orthogonal matrix $P$, and thus the Q factor computed 
by the sketched QR decomposition.
To this end, we first recall a result from \cite{burke2023gmres}, where
we use the notation
$$
\cos\angle(u,v) =\frac{u^\T v}{\|u\|_2\, \|v\|_2},
$$
for the cosine of the angle between two nonzero vectors.

\begin{lemma}({\rm \cite[Lemma 4.3]{burke2023gmres}})\label{lem:Sangle}
Let $S\in\mathbb{R}^{s\times m}$ be such that~\eqref{eq:def_subspace_emb} 
holds for two vectors $u,v\in\mathbb{R}^{m}$ and $\varepsilon>0$. Then with high
probability it holds that
$$
\frac{\cos\angle(u,v)-\varepsilon}{1+\varepsilon} \le
\cos\angle(Su,Sv)
\le \frac{\cos\angle(u,v)+\varepsilon}{1-\varepsilon} .
$$
\end{lemma}

Now, let $P=[p_1, \ldots, p_n]$ be
such that $(Sp_i)^\T (Sp_j)=\delta_{ij}$, where $\delta_{ij}$ denotes the Kronecker delta.
Then Lemma~\ref{lem:Sangle} implies
that 
$$
i\ne j,\,\, (Sp_i)^\T (Sp_j)=0 \,\, \Rightarrow \,\,
|\cos\angle(p_i,p_j)|\le \varepsilon .
$$
Therefore, the value $\varepsilon$ provides an upper bound also on the distortion of the cosine of the angle between two 
columns of an $S^\T S$-orthogonal matrix. 
More precise bounds on the distance from Euclidean orthogonality are given
in section~\ref{sec:distorth}.
%This can help in achieving more insights about the geometric interpretation of sketching techniques. 
%{\color{red}To the best of our knowledge, this result appears to be new.}
%
To this end, in the following section we discuss 
the $S^\T S$-SVD factorization, which will be used as a strategic tool 
for measuring the distance between subspaces, in agreement with the role of the
standard SVD \cite[Section 2.5]{Golub.VLoan.13}.
%that will be useful also to bound the distance from (Frobenius and Euclidean)
%orthogonality of $S^\T S$-orthogonal matrices.

 %%%%%%%%%%%%%%%%%%%%%%%%%%%%%%%%%%%%%%
 \section{The $S^{\top} S$-SVD}\label{SS-SVD}
In~\cite[Theorem 3]{VanLoan76}, Van Loan derived a generalization of the SVD of a given matrix in 
terms of nonstandard inner products. %We specialize that definition to our context.

\begin{definition}[{\cite[Definition 3]{VanLoan76}}]\label{Def:SSsingularvalues}
{%\color{red}
 Given $A\in\mathbb{R}^{m\times n}$, $m\geq n$, and a symmetric and
positive definite matrix ${\cal H} \in {\mathbb R}^{m\times m}$,
the ${\cal H},2$-singular values of $A$ are the elements of the following set
% an $\varepsilon$-subspace embedding $S\in\mathbb{R}^{s\times m}$ of $\text{Range}(A)$, 
%the $S^\T S,2$-singular values of $A$ are the elements of the following set
 %%
 \begin{equation}\label{eq:SSsingular_values_def}
  \mu_{{\cal H},2}(A):=\left\{\mu\geq 0 \text{ s. t. } \mu 
\text{ is a stationary {\color{black}value} of }\; \frac{\|Ax\|_{\cal H}}{\|x\|_2},\; x\neq0\right\}.
%  \mu_{S^\T S,2}(A):=\left\{\mu\geq 0 \text{ s. t. } \mu \text{ is a stationary point of }\; \frac{\|Ax\|_{S^\T S}}{\|x\|_2},\; x\neq0\right\}.
 \end{equation}
}
\end{definition}

%For the sake of brevity, in the following we will talk about $S^\T S$ quantities, in place of $S^\T S,2$, omitting the dependency on the 2-norm. 
%For instance, we write $\mu_{S^\T S}(A)=\mu_{S^\T S,2}(A)$.

%{\color{red}Explain this def in the context of Van Loan's theorem. DP. at this point I'd delete this definition and leave proposition 1 below.}

In our context, $S^\T S$ is not positive definite from a deterministic viewpoint.
Nonetheless, we can still formalize an SVD where the left singular vectors
are $S^\T S$-orthogonal{\color{black}, and the factorization
holds} with high probability.

%{\color{red}VAL: We should refer to \cite[Th.3]{VanLoan76},
%from which Th.1 seems to trivially follow, under the new hypothesis that the
%inner product matrix is  semidef.}

\iffalse
\begin{theorem}\label{th:SSsvd}
%\begin{theorem}[{\color{red}$S^\T S$-SVD~\cite[Theorem 3]{VanLoan76}}]\label{th:SSsvd}
 Given $A\in\mathbb{R}^{m\times n}$, $m\geq n$, and a {\color{black}
 randomized 
$(\varepsilon,\delta,\text{rank}(A))$}-subspace embedding $S\in\mathbb{R}^{s\times m}$ 
of $\text{Range}(A)$, $s\geq\text{rank}(A)$,
there exist\st{, with high probability,} an $S^\T S$-orthogonal matrix
$W\in\mathbb{R}^{m\times r}$, an
 orthogonal matrix $V\in\mathbb{R}^{n\times r}$, $r=\min\{s,n\}$, such that
$A$ can be written {\color{red}with high probability} as
 %%
 \begin{equation}\label{eq:SS-SVD}
 A=W\Theta V^\T, \quad \Theta=\begin{bmatrix}
       \theta_1&&\\
       &\ddots&\\
       &&\theta_r\\
       \end{bmatrix}\in\mathbb{R}^{r\times r},\quad \theta_1\geq\ldots\geq\theta_r\geq 0 .
\end{equation}
%
%%
%where $\mu_{S^\T S}(A)=\{\theta_1,\ldots,\theta_n\}$.
\end{theorem}
\fi

\begin{theorem}\label{th:SSsvd}
{\color{black}
%\begin{theorem}[{\color{red}$S^\T S$-SVD~\cite[Theorem 3]{VanLoan76}}]\label{th:SSsvd}
 Let $A\in\mathbb{R}^{m\times n}$, $m\geq n$, and let $S\in\mathbb{R}^{s\times m}$ be a {randomized 
$(\varepsilon,\delta,\text{rank}(A))$}-subspace embedding %$S\in\mathbb{R}^{s\times m}$ 
of $\text{Range}(A)$, $s\geq\text{rank}(A)$. Then
there exist an $S^\T S$-orthogonal matrix
$W\in\mathbb{R}^{m\times r}$, an
 orthogonal matrix $V\in\mathbb{R}^{n\times r}$ with $r=\min\{s,n\}$, and a diagonal matrix $\Theta={\rm diag}(\theta_1, \ldots, \theta_r)$ with  $\theta_1\geq\ldots\geq\theta_r\geq 0$ such that 
$A$ can be written  as 
\begin{equation}\label{eq:SS-SVD}
 A=W\Theta V^\T ,
\end{equation}
where the equality holds with probability $1-\delta$.
%$A=W\Theta V^\T$.
 %%
 %\begin{equation}\label{eq:SS-SVD}
 %A=W\Theta V^\T, \quad \Theta=\begin{bmatrix}
 %      \theta_1&&\\
 %      &\ddots&\\
 %      &&\theta_r\\
 %      \end{bmatrix}\in\mathbb{R}^{r\times r},\quad %\theta_1\geq\ldots\geq\theta_r\geq 0 .
%\end{equation}
}
\end{theorem}

%{\color{red}ADD: justification in terms of stationary points?}
{\color{black}
\begin{proof}
    The result directly comes from 
    \cite[Theorem 3]{VanLoan76} as long as~\eqref{eq:def_subspace_emb} holds true. Indeed, in this case $S^\T S$ is positive definite on $\text{Range}(A)$. On the other hand, the property~\eqref{eq:def_subspace_emb} holds with probability $1-\delta$ which means that the same happens for our factorization~\eqref{eq:SS-SVD}. $\hfill \square$    
\end{proof}
The next corollary reports a first characterization of the $S^TS$-singular values $\theta_i$'s; its proof follows the steps of that of
Theorem 3 in \cite{VanLoan76}.

\begin{corollary}\label{prop_stationaryvalues}
    The nonegative scalars $\theta_i$, $i=1,\ldots,r$, in Theorem~\ref{th:SSsvd} are stationary values of the function $\mu(x)=\|Ax\|_{S^\T S}/\|x\|_2$ with $x\neq 0$.
\end{corollary}
\begin{proof}
We consider the equivalent functional 
$\eta(x)=\|Ax\|_{S^\T S}^2$, with the constraint $\|x\|_2=1$, and define
the Lagrangian
$L(x,\lambda)=\|Ax\|_{S^\T S}^2-\lambda(\|x\|_2^2-1)=x^\T A^\T S^\T SAx-\lambda(x^\T x-1)$. Stationary points are determined by taking derivatives, and they are the zeros of 
${\rm det}(A^\T S^\T SA -\mu^2 I)=0$.

Let $SA=U_1\widehat \Theta V^\T$, $U_1\in\mathbb{R}^{s\times r}$, $\widehat \Theta\in\mathbb{R}^{r\times n}$, $V\in\mathbb{R}^{n\times n}$, denote the standard SVD of $SA$. If $s\geq n$, then $r=n$ and $\widehat \Theta=\Theta$. Otherwise, $r=s$ and $\widehat \Theta=[\Theta, 0]$.
Therefore, $A^\T S^\T SA = V \widehat \Theta^2 V^T$,
where $\widehat \Theta^2={\rm diag}(\Theta^2,0)$. Hence
$$
{\rm det}(A^\T S^\T SA -\mu^2 I)=
{\rm det}(\widehat \Theta^2 -\mu^2 I)=\prod_{i=1}^n (\widehat \theta_i^2-\mu^2).
$$
so that the first $r$ roots coincide with the $\theta_i$'s. $\hfill \square$
%
%
%In any case, by performing the change of variable $y=V^\T x$, the eigenvalue problem above can be written as 
%$$\left\{\begin{array}{rlr}
%\widehat\Theta^\T \widehat\Theta y&=&\lambda y,\\
%  y^\T y   &=&1.  \\
%\end{array}\right.$$
%The solutions of this problem are given by the eigenpairs $(\theta_i^2,e_i)$, $i=1,\ldots,r$, and $(0,e_i)$, $i=r+1,\ldots,n$,
%so that the stationary points of $\mu(x)$ are the vectors $x_i=Ve_i$ that achieve the stationary values
%$$\mu(x_i)=\frac{\sqrt{x_i^\T A^\T S^\T SAx_i}}{\sqrt{x_i^\T x_i}}=\sqrt{\theta_i^2}=\theta_i,\quad i=1,\ldots,r, \quad \text{and}\quad \mu(x_i)=0,\quad i=r+1,\ldots,n.$$
\end{proof}
}

{\color{black} Returning to Theorem \ref{th:SSsvd}, we notice} that since $S$ is a subspace embedding for Range($A$), $s$ will always be required
to be larger than the rank of $A$, even if this is less than $n$.
The next remark highlights the implications of possible null $\theta_i$s.

\begin{remark}
 Since $S\in\mathbb{R}^{s\times m}$ is asked to be a {\color{black}randomized $(\varepsilon,\delta,\text{rank}(A))$}-subspace embedding of $\text{Range}(A)$, the sketching dimension $s$ should be selected in terms of $\text{rank}(A)$. For simplicity,  $n${\color{black}, the number of columns of $A$, can be used instead of} $\text{rank}(A)$. While this strategy certainly works, it is important to stress that it can lead to
unnecessarily large values of $s$, especially in case of a 
low-rank $A$, which can significantly hinder the computational performance of the approach.
We refer to Example~\ref{ex:1.1bis} below for an example with a low-rank square matrix.
\end{remark}

\vskip 0.1in
The factors in the $S^\T S$-SVD of Theorem~\ref{th:SSsvd}
 can be computed by means of the following steps:

\vskip 0.1in
\begin{enumerate}
\item Compute the SVD of $SA=U_1 \Theta V^\T$
\vskip 0.1in
\item Define $W:=A(\Theta V^\T)^{\dagger}$
\end{enumerate}
\vskip 0.1in

{\color{black} Next, a} simple Matlab implementation  for $A$ full rank and equivalent to the procedure above is reported.

\vskip 0.1in
\begin{verbatim}
       [~,R]=qr(S*A,0);                           
       [~,Theta,V]=svd(R);                          
       W=A*(V/Theta);                             
\end{verbatim}
\vskip 0.1in

Notice that in the {\color{black} commands} above the only operations whose cost depends on the dimensionality $m$ of the problem are the application of the sketching $SA$ and the {\color{black}retrieval} of the left singular vectors $W=AV\Theta^{\dagger}$. Both these steps can take advantage of the possible sparsity of $A$.
In all reported experiments, we computed the $S^\T S$-SVD by the above commands.

This procedure {\color{black} is closely related} to the approach derived
in~\cite{Gilbertetal.14}, where, however, 
 the authors restricted the use of the term ``sketched SVD'' of $A$ to
 the SVD of $SA\in\mathbb{R}^{s\times n}$. 
%This shares some features with our $S^\T S$-SVD. In particular, the $S^\T S$-singular values $\theta_j$'s of $A$ and the (standard) singular values of $SA$ are the same. See also Theorem~\ref{Th:boundsingularvalues}. 
%The left singular matrix $U_1$ of $SA$ has $s$ rows that have a
% difficult interpretation. 
 {\color{black} With our approach, instead, the left singular
 matrix $W$ yields an $S^\T S$-orthogonal basis for the range of $A$.
 %We overcome this issue by directly working with $W$,
%adopting the concept of $S^\T S$-orthogonality\todo{needs some rewording}.
We thus refrain from using the adjective ``sketched'' 
for~\eqref{eq:SS-SVD}, as the latter refers to a
decomposition of the whole matrix $A$.}
%, which seems to imply 
%some sort of inexactness. The $S^\T S$-SVD is an \emph{exact} factorization of $A${\color{red}, with high probability}.

\begin{remark}\label{rem:invariance}
%We explicitly observe that 
{\color{black}If}
 $W$ is $S^\T S$-orthogonal and $U$ is square 
in the Euclidean sense, then $WU$ is still $S^\T S$-orthogonal.
The same holds if $U$ is tall and orthogonal, {\color{black}although}
$S^\T S$-orthogonality will be in terms of the number of columns of $U$.
\end{remark}

% stemming from Definition~\ref{Def:SSsingularvalues}.
%{\color{magenta}VS:The def uses an spd matrix, here we need to introduce
%the ``high probability'' of a semidef matrix, so it is not the same def.}

\begin{remark}
The $S^\T S$-SVD of $A$ can also be determined by means of the randomized 
QR factorization of section~\ref{Subspace embeddings and randomized QR}:
compute
$A=QR$, with $Q^\T S^\T S Q=I$. % and $R$ upper triangular.
Then, compute the standard SVD of $R$,
namely $R=U\Theta V^\T$, so that the $S^\T S$-SVD of $A$ is given by
$$A=W\Theta V^\T,\quad W:=QU.$$
{\color{black}Using} the randomized QR, this implementation of the $S^\T S$-SVD may turn out to be more robust than the one illustrated above {\color{black}which could be prone to numerical instabilities due to the computation of $\Theta^\dagger$}. On the other hand, the former is in general more expensive due to the computation of $Q$.
We also stress the \emph{single-pass} nature of this implementation of the $S^\T S$-SVD, which
allows us to access $A$ only once, as required by certain
data streaming models. At the same time, the $S^\T S$-SVD  is still able to compute a decomposition with ($S^\T S$-)orthogonal factors, unlike other state-of-the-art 
randomized methods; see, e.g., the discussion in ~\cite[section 5.5]{Halko2010}.
\end{remark}

%\begin{remark}
By employing the $S^\T S$-SVD we can provide alternative definitions of the $S^\T S$-matrix norms in Definition~\ref{Def:SSnorms}, that is
$$\SF{A}^2=\sum_{i=1}^{n}\theta_i^2,\quad\text{and}\quad \Sdue{A}=\theta_1.$$
% \end{remark}

%\begin{definition}
%Given $A\in\mathbb{R}^{m\times n}$, $m\geq n$, if $\theta_i,\ldots,\theta_n$ are its sketched singular values, then the sketched norm of $A$ is given by
%%
%$$\|A\|_{S^\T S}^2:=\sum_{i=1}^{n}\theta_i^2=\text{trace}(A^T S^T S).$$
%%
%\end{definition}

As its standard counterpart, also the $S^\T S$-SVD and the $S^\T S$-singular values $\theta_i$'s fulfill certain optimality conditions as long as the latter are formulated in the right norm i.e., the $S^\T S$-norm.

\begin{theorem}\label{th:SVDbounds}
Let $A=W\Theta V^\T$ 
with $\Theta=\text{diag}(\theta_1,\ldots,\theta_n)$ 
be the $S^\T S$-SVD of $A\in\mathbb{R}^{m\times n}$, $m\geq n$. 
Let $W_k\in\mathbb{R}^{m\times k}$, $V_k\in\mathbb{R}^{n\times k}$ collect the 
first $k$ columns of $W$ and $V$, respectively, and 
$\Theta_k$ be the square top left leading part of $\Theta$.
Let $S\in\mathbb{R}^{s\times m}$ be a {\color{black}randomized $(\varepsilon,\delta,n)$}-subspace embedding of any subspace
of ${\mathbb R}^{m}$ of dimension $n$.
Then with high probability
\begin{equation}\label{eq:truncatedSVD}
 \min_{\substack{B\in\mathbb{R}^{m\times n}\\ \text{rank}(B)=k}}
\SF{A-B}^2={\sum_{i=k+1}^n\theta_{i}^2},\qquad
W_k\Theta_k V_k^\T =
\argmin_{\substack{B\in\mathbb{R}^{m\times n}\\ \text{rank}(B)=k}}
\SF{A-B} ,
\end{equation}
\begin{equation}\label{eq:truncatedSVD2}
 \min_{\substack{B\in\mathbb{R}^{m\times n}\\ \text{rank}(B)=k}}
\Sdue{A-B}^2={\theta_{k+1}},\qquad
W_k\Theta_k V_k^\T =
\argmin_{\substack{B\in\mathbb{R}^{m\times n}\\ \text{rank}(B)=k}}
\Sdue{A-B} ,
\end{equation}
and
\begin{equation}\label{eq:optimaltheta}
 \theta_k=\max_{\mathcal{U}, \text{dim}(\mathcal{U})=k} \;\min_{\substack{x\in\mathcal{U}\\ \|x\|_2=1}}\|Ax\|_{S^\T S}=\min_{\mathcal{U}, \text{dim}(\mathcal{U})=n-k+1} \;\max_{\substack{x\in\mathcal{U}\\ \|x\|_2=1}}\|Ax\|_{S^\T S}.
\end{equation}
\end{theorem}

\begin{proof}
All these optimality results can be shown by mimicking the proofs of the corresponding results for standard SVD and recalling that the left $S^\T S$-singular vectors are $S^\T S$-orthogonal.$\hfill \square$

\end{proof}

In the analysis that follows, it will be important to relate the $S^\T S$-singular values 
of $A$ to their standard counterparts. 
%From the definition of the set~\eqref{eq:SSsingular_values_def}, we can write
%%
%$$\mu_{S^T S}(A)=\left\{\mu\geq 0 \text{ s. t. } \mu \text{ is a stationary point of }\; \|SAx\|_2,\; \|x\|_2=1\right\}.$$
%%
Since $S$ is a {\color{black}randomized $(\varepsilon,\delta,\text{rank}(A))$}-subspace embedding for $\text{Range}(A)$, we know that
$$(1-\varepsilon)\|Ax\|_2^2\leq\|SAx\|_2^2\leq(1+\varepsilon)\|Ax\|_2^2,$$
so that
$$(1-\varepsilon)\sigma_n^2\leq\|SAx\|^2_2\leq(1+\varepsilon)\sigma_1^2,$$
where $\sigma_1$ and $\sigma_n$ are the largest and smallest (standard) singular values of $A$, respectively. 
This means that all the $S^\T S$-singular values of $A$ are included in the interval 
$[\sqrt{1-\varepsilon}\cdot\sigma_n,\sqrt{1+\varepsilon}\cdot\sigma_1]$.
More accurate bounds can be obtained by monitoring the change {\color{black} in} each singular value.
Such an analysis, together with {\color{black}the} corresponding bounds for the right singular vectors
was carried out in \cite{Gilbertetal.14}, and since then it has been rediscovered a few times, 
also with different proofs, 
% The proof in the latter manuscript can be significantly reduced by employing the optimality of the $S^T S$-singular values~\eqref{eq:optimaltheta} and the properties of
%$\varepsilon$-subspace embeddings as done 
see, e.g., ~\cite[Theorem 2.2]{RandrankrevealingQR}.

\begin{theorem}{\rm (\cite{Gilbertetal.14})} \label{Th:boundsingularvalues} 
Let $S$ be a {\color{black} randomized $(\varepsilon,\delta,\text{rank}(A))$}-subspace embedding for Range($A$), and
 let $A=W\Theta V^\T$ be the $S^\T S$-SVD 
of $A\in\mathbb{R}^{m\times n}$, $m\geq n$, 
and $A=U\Sigma Y^\T$, $\Sigma=\text{diag}(\sigma_1,\ldots,\sigma_n)$,
its standard SVD. Then with high probability, 
 \begin{equation}\label{eq:boundsingularvalues}
\sqrt{1-\varepsilon}\cdot \sigma_k
\leq
\theta_k
\leq\sqrt{1+\varepsilon}\cdot\sigma_k,\quad\text{for all }k=1,\ldots,n.
 \end{equation}
\end{theorem}

%\begin{proof}
 %The proof follows from the optimality of the sketched
%singular values~\eqref{eq:optimaltheta} and %the properties of
%$\epsilon$-subspace embeddings. In particular, it holds
%%
%\begin{align*}
%\theta_k=&\,\max_{\mathcal{U}, \text{dim}(\mathcal{U})=k} \;\min_{\substack{x\in\mathcal{U}\\ \|x\|_2=1}}\|Ax\|_{S^T S}=\max_{\mathcal{U}, \text{dim}(\mathcal{U})=k} \;\min_{\substack{x\in\mathcal{U}\\ \|x\|_2=1}}\|SAx\|_2\\
%\leq&\, (1+\varepsilon)\max_{\mathcal{U}, \text{dim}(\mathcal{U})=k} \;\min_{\substack{x\in\mathcal{U}\\ \|x\|_2=1}}\|Ax\|_2=
%(1+\varepsilon)\sigma_k.
%\end{align*}
%%
%Similarly, $\theta_k\geq(1-\varepsilon)\sigma_k$,
% and the result follows.
%\end{proof}

Theorem~\ref{Th:boundsingularvalues} is an important 
asset in the analysis of sketched numerical procedures. For 
instance, the rank of $A$ can be read by looking at the $S^\T S$-singular values. Indeed, $\theta_k=0$ if and only if $\sigma_k=0$, with high probability.
Moreover,  the \emph{truncated} 
$S^\T S$-SVD $W_k\Theta_kV_k^\T$ in~\eqref{eq:truncatedSVD} satisfies
$$
\SF{A-W_k\Theta_kV_k^\T}^2=
\sum_{i=k+1}^n\theta_i^2\leq (1+\epsilon)\sum_{i=k+1}^n\sigma_i^2,\; \text{and}\; \Sdue{A-W_k\Theta_kV_k^\T}=
\theta_{k+1}\leq \sqrt{1+\epsilon}\cdot\sigma_{k+1}.$$
Therefore, the truncated $S^\T S$-SVD, which is optimal in the 
$S^\T S$-norm, attains an error that with high probability is a small 
multiple of the optimal error in the Frobenius and Euclidean norm. 
%{\color{red}Moreover, the
%construction of the $S^\T S$-SVD shares the computational advantages 
%of using the sketched QR algorithm over its standard counterpart.}\todo{VS:QR adv no longer needed?}
This opens to the design of fast, suboptimal low-rank approximations with 
$S^\T S$-orthogonal factors.

%In the following example we provide a numerical validation of some of the findings presented in this section.
It is also interesting to compare the $S^\T S$-SVD with the 
randomized SVD~\cite{Halko2010}, which relies on a
%The latter does not rely on subspace embeddings but rather on the 
randomized
range finder algorithm in place of subspace embeddings. 
%In particular, 
Given a target rank $k$ and an oversampling parameter $\ell$, the range finder first draws a sketching matrix
$\Omega\in\mathbb{R}^{n\times (k+\ell)}$, which does not necessarily fulfill any subspace embedding properties, and it then constructs a matrix $Q\in\mathbb{R}^{m\times (k+\ell)}$ whose columns form an orthonormal basis of $\text{Range}(A\Omega)$. The SVD-like approximation provided by the randomized SVD algorithm is $\widehat A:=(Q\widehat U)\widehat\Sigma\widehat V^\T\approx A$ where $\widehat U\widehat\Sigma\widehat V^\T$ is the standard SVD of $Q^\T A$. By partitioning the right singular vectors of $A$ as $V=[V_1,V_2]$, $V_1\in\mathbb{R}^{n\times k}$, $V_2\in\mathbb{R}^{n\times (n-k)}$ and by denoting $\Omega_1:=V_1^\T\Omega$, $\Omega_2:=V_2^\T\Omega$, under the 
assumptions $\text{rank}(\Omega_1)=k$ and $\gamma_k:=\sigma_{k+1}/\sigma_k<1$, in~\cite[Theorem 9]{RSVD_singularvalues2} 
it was shown that\footnote{Note that in the original version 
of~\cite[Theorem 9]{RSVD_singularvalues2} the term
$\gamma_j^{4q+2}$ appears in place of $\gamma_j^2$ as we report. The scalar $q$ consists of the number of 
performed subspace iterations. Here 
we assume $q=0$, namely the plain randomized SVD algorithm is adopted.}
$$\frac{\sigma_j}{\sqrt{1+\gamma_j^2\|\Omega_2\Omega_1^\dagger\|_2^2}}\leq\widehat \sigma_j\leq\sigma_j,\quad\text{for all }j=1,\ldots,k,$$
where $\widehat\sigma_j$ denotes the $j$th singular values of $\widehat A$. 
The width of the interval containing the ratio $\widehat \sigma_j/\sigma_j$ depends on the gap between
the corresponding consecutive singular values of $A$ and on the quality of the random matrix $\Omega$.
In our case, the corresponding ratio for
the $S^\T S$-singular value $\theta_j$ lies in the interval $[\sqrt{1-\varepsilon}, \sqrt{1+\varepsilon}]$,
 which is related to the considered probabilistic confidence.
A more detailed comparison is postponed to later research.
% of the 
%It would be interesting to relate the bounds above with the 
%ones in Theorem~\ref{Th:boundsingularvalues}. We leave this fascinating matter to be studied elsewhere.
%}

%%%%%%%%%%%%%%%%%%%%%%%%%%%%
\subsection{On the applicability of the $S^\T S$-{\color{black}SVD}}\label{On the applicability}
In many applications the singular values of the given matrix $A$ 
are characterized by particular decay properties or by the presence
of interval gaps.
%has a significant decay in its singular values as it comes from, e.g., a low-rank model. 
In the latter case, for instance, the gap may correspond to a very
ill-conditioned $A$; truncation procedures have been classically
 adopted to eliminate what can be recognized as noisy data,
see, e.g., \cite{VarahApril1973a}.
%that approximate truncate 
%and it is important to regularize it by cutting off its (very) low frequencies. 
In these scenarios,
% the $S^\T S$-SVD can be helpful. In particular, it is possible to compute only 
the $S^\T S$-singular values $\theta_i$s of $A$ can be computed at low cost, 
and analyzed in place of the original ones.
%and study their behavior. 
Thanks to~\eqref{eq:boundsingularvalues}, the qualitative behavior of the
$\theta_i$s will be the same as that of the singular values $\sigma_i$s of $A$.
% will have the same trend displayed by the $\theta_i$s, quantitatively speaking. 
For example, a distribution gap of the $\sigma_i$s can be captured by inspecting
a similar gap in the $\theta_i$s (see Figure~\ref{fig:SScompare} below), and thus 
determining the number of singular values to be retained.
Such a strategy can be employed to  determine the target rank in 
probabilistic or deterministic procedures that seek a low rank
approximation of $A$ of given rank, such as truncated SVD, CUR, and 
Nystr\"om type approximations \cite{doi:10.1137/140977898},%
\cite{Halko2010},\cite{Martinsson_Tropp_2020}.
Moreover, the $S^\T S$-singular values $\theta_i$s can  
be exploited to determine a priori
the truncation threshold of the solution in regularized 
 linear least squares problems associated with ill-posed inverse problems; see, e.g.,
\cite{Hansen.book.10}.
% we expect $A$ to present a large gap in its $\sigma_i$s but we do not know where this is located, we can extrapolate this information by looking at the $\theta_i$s as they showcase the same gap with high probability.
%Therefore, the $S^TS$-SVD can be employed as a tool to analyze the singular values' trend of a matrix at low cost. This piece of information can be further exploited in computation by performing, e.g., (standard) low-rank approximations or employing low-rank methods where providing the rank of the underlying model, or an upper bound thereof, is crucial.

The next example illustrates that the $S^\T S$-singular values 
are indeed able to 
capture the low numerical rank occurring in the original data.
%encode a gap similar to the one presented by their standard counterparts.

  \begin{example}\label{ex:1.1bis}
{\rm
We consider the Cauchy matrix $C\in\mathbb{R}^{n\times n}$, $n=5\,000$, 
defined entry-wise as $C_{i,j}=1/(x_i+y_j)$ where $x_i$ and $y_j$ are taken 
as the $i$th and $j$th nodes of a grid of $n$ equidistant 
points in $[2,100]$ and $[-1\,000,-500]$, respectively. 
Cauchy matrices present a fast decay in their singular values 
if, e.g., the values $\{x_i\}$ and $\{y_j\}$ come from disjoint 
sets; see, e.g.,~\cite[Section 4.1]{BecT17}. 
Thus, the considered matrix $C$ is numerically low-rank. We show that 
the $S^\T S$-SVD amounts to a practical, efficient way for 
estimating its numerical rank.

As sketching $S\in\mathbb{R}^{s\times m}$ we consider the following subsampled trigonometric transformation
\begin{equation}\label{eq:sketching_trig}
S=  \sqrt{\frac{m}{s}}DFE,
\end{equation}
where $E\in\mathbb{R}^{m\times m}$ is a diagonal matrix with Rademacher entries
(i.e., the diagonal entries are randomly chosen as $\pm 1$ with equal
probability), $D\in\mathbb{R}^{s\times m}$
 contains $s$ randomly selected rows of
the identity matrix, and $F\in\mathbb{R}^{m\times m}$ is the discrete cosine transform.
Due to our rather basic implementation, the cost of applying the matrix $S$ to an $m \times n$ matrix is $\mathcal O(mn \log m)$ operations\footnote{The action of $F$ is computed via the matlab function {\tt dct}.}.

It was shown in \cite{Tropp2011}
 that if $s=\mathcal{O}(\varepsilon^{-2}(k+\log\frac{m}{\delta})\log\frac{k}{\delta})$, then~\eqref{eq:sketching_trig} is an oblivious $\varepsilon$-subspace
embedding for any $k$-dimensional subspace of $\mathbb{R}^m$.
Nonetheless, numerical evidence suggests that selecting the
smaller sketching dimension 
$s=\mathcal{O}(\varepsilon^{-2}\frac{k}{\delta})$ works well in 
practice; see, e.g.,~\cite[Section~9]{Martinsson_Tropp_2020}.

For the given $C$, although it is known that the rank is low,
the actual dimension of $\text{Range}(C)$ is not known a-priori.
We thus select rather moderate values of $s$ and in 
Figure~\ref{fig:SScompare} we report the first 40 singular values
$\sigma_i$s of $C$ computed by {\tt svd} (blue stars), 
the first $\ell:=\min\{s,40\}$ $S^\T S$-singular values $\theta_i$s (red circles),
and the first $\ell$ standard singular values computed by {\tt svds}
(black diamonds) for %different values of $s$: 
$s=30$ (left) and $s=60$ (right).

%{\color{red}VS: i plot in Fig.1 sono molto piccoli, rispetto allo spazio
%che hanno. ridurrei l'asse delle ordinate e mostrerei solo i primi 40 sing.vals.
%l'agenda potrebbe essere piu grande, cosi come i font. (ho provato a caricare
%la figura *.fig ma non mi fa fare editing)}
%

%  \begin{figure}[t]
%  \centering
%  \begin{minipage}{0.45\textwidth}
%  \centering
%  \hspace{-1cm}
%  \includegraphics[scale=.47]{cauchy_ss_5rank.pdf}
%  \end{minipage}~\begin{minipage}{0.5\textwidth}
%  \centering
% % \hspace{-.8cm}
%  \includegraphics[scale=.47]{cauchy_ss_10rank.pdf}
%\end{minipage}  

  \begin{figure}[t]
  \centering
%  \begin{minipage}{0.45\textwidth}
%  \centering
%  \hspace{-1cm}
  \includegraphics[scale=.45]{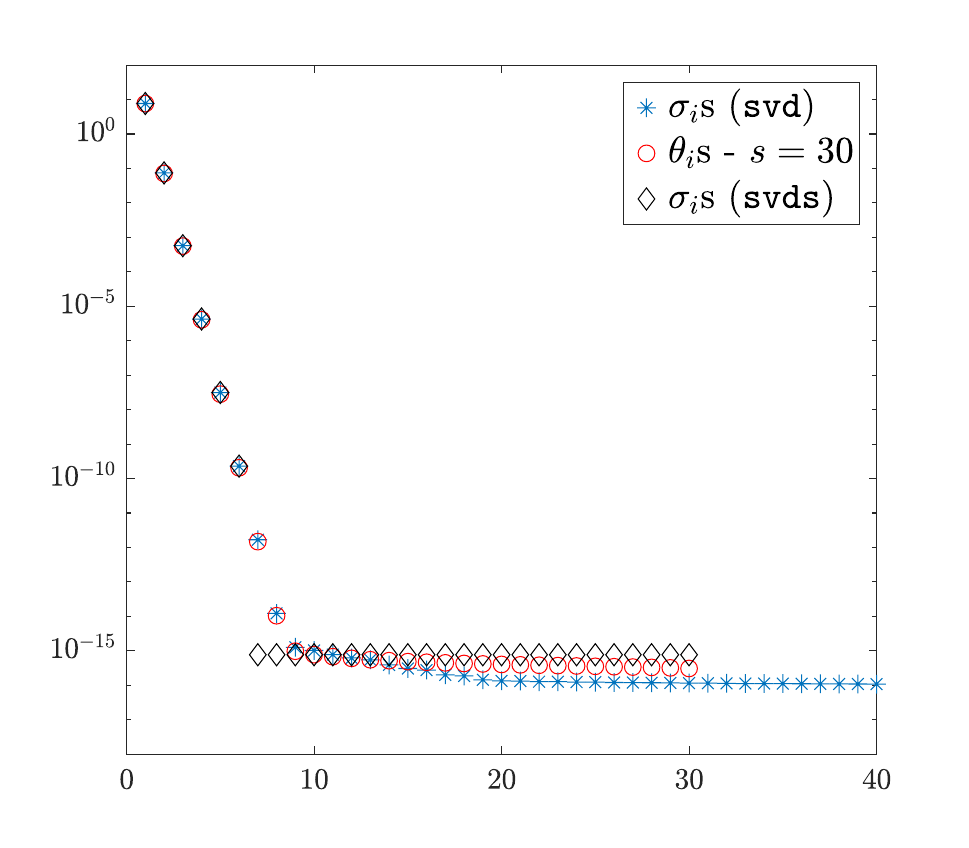}
%  \end{minipage}~\begin{minipage}{0.5\textwidth}
%  \centering
 % \hspace{-.8cm}
  \includegraphics[scale=.45]{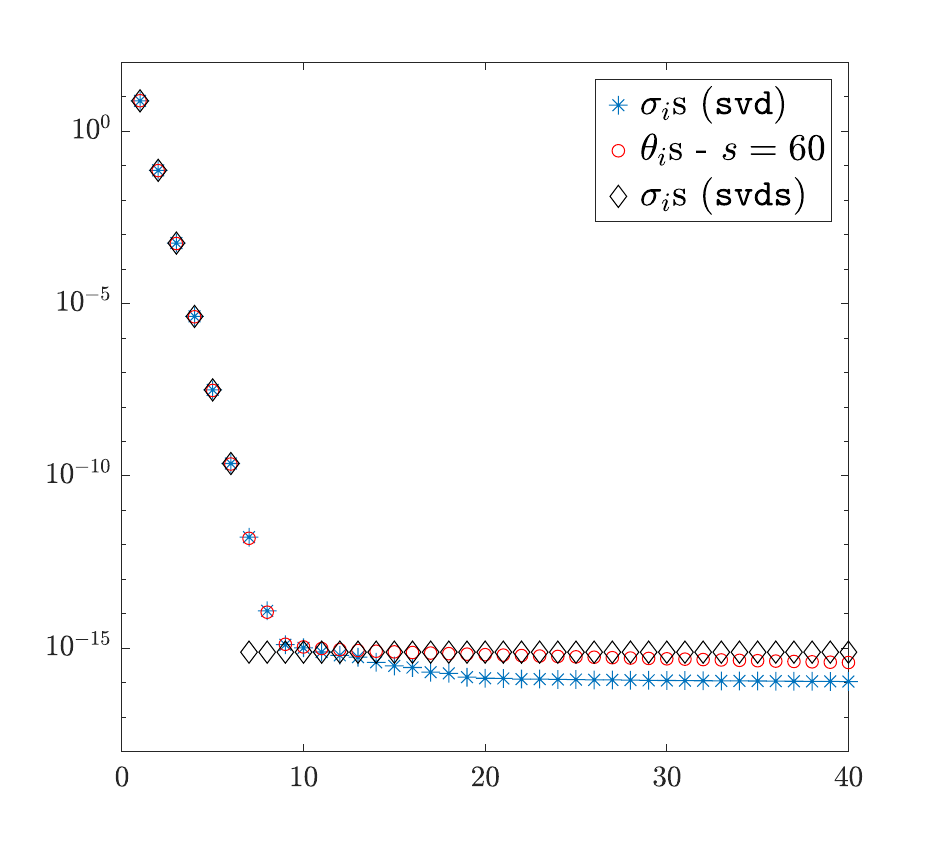}
%\end{minipage}

   \caption{Example~\ref{ex:1.1bis}.
   First 40 standard singular values $\sigma_i$s computed by {\tt svd} (blue stars), first $\ell$ $S^\T S$-singular values $\theta_i$'s (red circles), and first $\ell$ standard singular values $\sigma_i$s computed by {\tt svds} (black diamond) for different values of the sketching dimension $s$ and $\ell=\min\{s,40\}$. Left: $s=30$. Right: $s=60$. The reported results have been averaged over 50 runs.
   \label{fig:SScompare}}
  \end{figure}

Figure~\ref{fig:SScompare} shows that the $S^\T S$-singular values 
are able to match rather well the true singular values of $A$, 
especially those larger than $10^{-15}$, for both the tested values of $s$. 
The $S^\T S$-SVD thus perfectly encodes the (standard) numerical rank of $C$. 
This does not really happen when running the memory-saving deterministic 
Matlab function
{\tt svds}\footnote{This is still the case when changing the default setting of {\tt svds} to, e.g., {\tt sigma=svds(A,s,'largest','Tolerance', 1e-15,'MaxIterations',2e3)}.}. 
Indeed, this routine is unable to catch $\sigma_7\approx 10^{-12}$ and 
$\sigma_8\approx 10^{-14}$. 
%whose magnitude is $10^{-12}$ and $10^{-14}$, respectively.

The computation of the $S^\T S$-singular values required only
$0.14$ and $0.15$ seconds for $s=30$ and $s=60$, respectively. 
These timings are one order of magnitude smaller than the 
running time devoted to the computation of all the standard 
singular values of\footnote{{\tt sigma=svd(C)}.} $C$, 
which amounts to 3.31 seconds. 
The matlab function {\tt svds} required 1.82 ($s=30$) and 3.6 ($s=60$) seconds 
to compute the first $s$ $\sigma_i$'s. 
In addition to being still larger than the running time required by our 
approach, these numbers show that the cost of {\tt svds} grows linearly with 
the parameter $s$. 
On the other hand, the cost of our approach is practically insensitive to $s$.

  We conclude by mentioning that also the randomized SVD algorithm with a subsampled trigonometric sketching matrix $\Omega\in\mathbb{R}^{n\times (k+\ell)}$, $k=s$, $\ell=5$, is able to match all the first eight singular values of $A$ with running times that are slightly larger than the ones attained by the $S^\T S$-SVD. This is mainly due to the explicit projection $Q^\T A$ performed by the former scheme.
\begin{flushright}
 $\diamond$
\end{flushright}
 }
\end{example}

\section{Orthogonality properties} \label{sec:distorth}
In using the sketched approach, a natural question is how far
 an $S^\T S$-orthogonal matrix is from being orthogonal in
%the closest orthogonal matrix in
the Frobenius and Euclidean sense.
The $S^\T S$-singular values can be employed for these purposes.
As a counterpart result, we will also derive bounds on the distance of standard 
orthogonal matrices from $S^\T S$-orthogonality.

\begin{proposition}\label{prop:orthP}
%Let $P$ solve problem (\ref{eq:main}).
Let $P\in\R^{m\times n}$ have $S^\T S$-orthonormal columns with $S$ being a {\color{black}randomized $(\varepsilon,\delta,\text{rank}(P))$}-subspace embedding for $\text{Range}(P)$.
Then,
with high probability,
$$
\|P^\T P-I\|_F\le \frac{\varepsilon}{1-\varepsilon}{{\sqrt{n}}}, \quad\text{and}\quad
\|P^\T P-I\|_2\le \frac{\varepsilon}{1-\varepsilon}.
$$
\end{proposition}

\begin{proof}
Since $P$ has $S^\T S$-orthonormal columns, {\color{black} all} $S^\T S$-singular values $\theta_j$ of $P$ are equal to 1. Now, let $P=U\Sigma V^\T$, $\Sigma=\text{diag}(\sigma_1,\ldots,\sigma_n)$, be the standard SVD of $P$.
Since $\sigma_j^2\leq 1/(1-\varepsilon)$ for all $j=1,\ldots,n$, it holds
$$
	\|P^\T P-I\|_2=\|\Sigma^2-I\|_2=\max_{i=1,\ldots,n}|\sigma_i^2-1|
\leq\frac{1}{1-\varepsilon}-1=\frac{\varepsilon}{1-\varepsilon} .
$$
Similarly,
\begin{align*}
 \|P^\T P-I\|_F^2=&\|\Sigma^2-I\|_F^2=\sum_{i=1}^n(\sigma_i^2-1)^2\leq\sum_{i=1}^n\left(
\frac{1}{1-\varepsilon}-1\right)^2=n\left(\frac{\varepsilon}{1-\varepsilon}\right)^2 .
\end{align*}
\end{proof}

The bounds of the {\color{black}above} proposition can be proved in different ways. For {\color{black}example}, directly using the
fundamental property in (\ref{eq:normS}), the orthogonality
$P^\T S^\T S P=I$ and the bounds on the 
2-norm of $P$, we can obtain the bounds of
$\|P^\T P-I\|_2$ by estimating  the inner product
$x^\T (P^\T P-I) x = x^\T P^\T(I-S^\T S) Px$.

We also observe that the previous bounds hold for any
 $S^\T S$-orthogonal matrix, thus also for the one stemming
from the randomized (sketched) QR decomposition~\cite{RandomGS}.

\begin{proposition} \label{prop:orthT}
Let $T\in\R^{m\times n}$ have orthonormal columns.
Let $S$ be a {\color{black}randomized $(\varepsilon,\delta,\text{rank}(T))$}-subspace embedding for $\text{Range}(T)$.
Then,
with high probability,
$$
\|T^\T S^\T S T-I\|_F\le {\varepsilon\sqrt{n}}, \quad\text{and}\quad
\|T^\T S^\T S T-I\|_2\le {\varepsilon}.
$$
\end{proposition}

\begin{proof}
Let $T=W \Theta V^\T$ be the $S^\T S$-singular value decomposition of $T$, and recall that the
standard singular values $\sigma_i$ of $T$ are all equal to 1.
	Then from $\|T^T S^\T S T-I\|_* = \|\Theta^2-I\|_*$ and using Theorem \ref{Th:boundsingularvalues}, it follows
$$
\|T^\T S^\T S T-I\|_2\le \max |\theta_i^2-1| \le \max |(1+\varepsilon)\sigma_i^2 - 1| = \varepsilon,
$$
while
$$
\|T^\T S^\T S T-I\|_F^2\le \sum_i |\theta_i^2-1|^2 \le \sum_i |(1+\varepsilon)\sigma_i^2 - 1|^2 = \varepsilon^2 n,
$$
from which both results follow.
\end{proof}

Again, using the property in (\ref{eq:normS}), the orthogonality bound
$\|T^\T S^\T S T-I\|_2\le {\varepsilon}$ could also be proved by noticing
that $\|T^\T S^\T S T-I\|_2= \|T^\T S^\T S T-T^\T T\|_2$ and that
$x^\T T^\T (S^\T S -I)T x = y^\T (S^\T S -I)y$ with $y=Tx$, where
$\|y\|=\|x\|$.

In the next example we illustrate what to expect in practice, in terms of
distance from full orthogonality, when relying on results such as that
in Proposition~\ref{prop:orthP}.

  \begin{table}[t!]
   \centering
   \begin{tabular}{rrrr}
    $s$ & $\|W^\T W-I\|_F$ &  $\|W^\T W-I\|_2$ & Time (s)\\
    \hline
    $55\log n$ & 16.84& 0.99 & 0.27\\
    $60\log n$ & 17.04& 0.99& 0.28\\
    $65 \log n$ & 17.13& 0.99& 0.28\\
    %$70 \log n$ & 7.87& 1.36& 1.16\\
   \end{tabular}
\caption{Example~\ref{ex:1.1}. Loss of orthogonality (Frobenius and Euclidean norm) 
of the first $n$ left $S^\T S$-singular left vectors of $A$ (i.e. the columns of $W$)
as $s$ varies. ``Time'' (in secs) is the running time to compute the $S^\T S$-SVD of $A$.
The reported results are averaged over 50 runs.  \label{tab:ex1}}
  \end{table}

  \begin{example}\label{ex:1.1}
{\rm
We consider the matrix $A\in\mathbb{R}^{m\times n}$, $m=300\,000$, $n=300$, 
generated by the matlab function {\tt A=sprand(m,n,0.003,1e-10)}. 
Therefore, only about 0.3\% of the entries of $A$ 
are non-null, and the nonzero values are taken from a uniform random distribution. Moreover, $\kappa(A)=10^{10}$.

% whereas the singular value distribution of $A_2$ displays a much slower decay.

For this example we change the nature of the sketching and consider
 $S\in\mathbb{R}^{s\times m}$ to be a Gaussian matrix.
The cost of computing $SA$ with $A\in\mathbb{R}^{m \times n}$ is 
now $\mathcal O(s\cdot\text{nnz}(A))$ flops where $\text{nnz}(A)$ denotes the number of nonzero entries of $A$. 
Moreover, to embed an $n$-dimensional subspace with failure probability $\delta$, it is sufficient to choose $s=\mathcal{O}(\varepsilon^{-2}\log n\log\frac{1}{\delta})$; see, e.g.,~\cite[Theorem 2]{Sarlos2006}.
%\todo{l'argomentaz e' contro gaussian? DP: tagliato. in effetti non c'è per forza bisogno di fare confronto. Possiamo semplicemente listare costi ed sketching dimension per la Gaussiana.}

%This is confirmed by the results reported in Figure~\ref{fig:svd}. Indeed, for both $A_1$ and $A_2$, when we select $s=2n$, the bound $|\sigma_i^2-\theta_i^2|/\sigma_i^2\leq 0.5$ does not hold for all $i$. This issue is fixed by increasing $s$.

  We are interested in illustrating the bounds in Proposition~\ref{prop:orthP}.
  In Table~\ref{tab:ex1}, for different values of $s$ we report $\|W^\T W-I\|_*$, $*=2,F$, where the columns of $W\in\mathbb{R}^{m\times n}$ are the first $n$ left $S^\T S$-singular vectors of $A$.
  If we set $\varepsilon=0.5$, as it is common when working with oblivious subspace embeddings, then $\varepsilon/(1-\varepsilon)=1$ so that the bounds in Proposition~\ref{prop:orthP} are satisfied whenever $\|W^\T W-I\|_F\leq\sqrt{300}\leq 18$ and $\|W^\T W-I\|_2\leq 1$.

  For $\varepsilon=0.5$ and $\delta=10^{-6}$, we should select
$s=\mathcal{O}(56\log n)$ for $S$ to be a {\color{black}randomized $(\varepsilon,\delta,n)$}-subspace embedding with probability $1-\delta$. From the results in Table~\ref{tab:ex1} we can see that, with this set of parameters, the bounds in Proposition~\ref{prop:orthP} are always fulfilled.
In Table~\ref{tab:ex1} we also document the
running times devoted to computing the $S^\T S$-SVD of $A$.
 We notice that these timings only mildly depend on the sketching 
dimension $s$, and they are rather competitive 
with the timing required to compute the standard SVD of $A$, which is 2.89 seconds
with the matlab {\tt svd} function\footnote{$\mathtt{[U,Sigma,V]=svd(A,0)}$.}. 
This is mainly due to the need of {\tt svd} to work with a matrix allocated in 
full format\footnote{The time required for changing $A$ to full format was
not included.}. 
For the chosen sketching this drawback does not affect the
 $S^\T S$-SVD routine, that can fully exploit the sparsity of $A$.
%
%{\color{red}Since for the $S^\T S$-SVD we employ a in-house implementation of Algorithm~\ref{alg:randGS}, to have fair comparisons we do the same also for the standard SVD. In particular, we implemented a basic Gram-Schmidt QR routine to compute the QR factorization of $A_1$. Then we compute the SVD of the triangular factor by the matlab function {\tt svd}. This implementation requires about 28 seconds to compute the SVD of $A_1$. }
%We must mention, however, that using the matlab function {\tt qr} to compute a skinny QR of $A_1$ would result in much lower computational timings.

% From the results reported in Table~\ref{tab:ex1} we can see that the running times for computing the $S^\T S$-SVD of $A_1$ are very competitive and, .
%
\begin{flushright}
 $\diamond$
\end{flushright}
 }

\end{example}

%%%%%%%%%%%%%%%%%%%%%%%%%%%%%%%%%%%%%
\section{The nearest {\color{black}$S^{\top\!}S$-}orthogonal matrix}\label{The nearest sketched orthogonal matrix}
With the goal of reducing the solution cost when dealing with large dimensional matrices, 
in~\cite{Bjorck} the authors derived an iterative algorithm for~\eqref{eq:main0}. 
With the same purposes, in this 
section we propose to explore the use of sketching techniques.
%pursue a different path. 
%In particular, in light of the great success that sketching procedures are having in reducing the cost of solving (vector) least squares problems, it is somehow natural to employ these tools also for treating~\eqref{eq:main0}. 
To this end, we first transform the problem by using the $S^\T S$-norm.
%In particular, 
Given a matrix $A\in\mathbb{R}^{m\times n}$, $m\geq n$, we
consider the following matrix minimization problem
\begin{equation}\label{eq:main}
 \min_{\substack{Q\in\mathbb{R}^{m\times n}
 \\ Q^\T S^\T S Q=I}}\|A-Q\|_{S,*},\quad *=2,F.
\end{equation}
Thanks to the $S^\T S$-SVD in section~\ref{SS-SVD},
 we are going to show the existence and uniqueness of its solution and relate 
the latter to the solution of the original problem~\eqref{eq:main0}.

%where $S\in\mathbb{R}^{s\times m}$ is a sketching matrix amounting
%to an $\varepsilon$-subspace embedding for $\text{Range}(A)$ and
%%
%$$\|B\|^2_{S^\T S}:=\text{trace}(B^\T S^T S),$$
%%
%is the $S^T S$-norm induced by the $S^T S$-inner product $\langle x,y\rangle_{S^T S}:=y^T S^T S$.

%We construct the matrix $P=WUV^T\in\mathbb{R}^{m\times n}$ as follows. 
%To construct the minimizer in (\ref{eq:main}),
% we compute the ``randomized'' QR factorization of $A$:
%we let $SA=GR$ be the reduced QR  decomposition of the sketched matrix $SA$,
%and then we define $W=AR^{-1}$. Hence, we have obtained the factorization
%%
%\begin{equation}\label{eqn:WR}
%A=WR,\quad {\rm with} \,\, W^\T S^\T S W=I,\;R\text{ upper triangular}.
%\end{equation}
%%
%Then we compute the deterministic SVD of $R\in\mathbb{R}^{n\times n}$, namely $R=U\Theta V^\T$.

Now let $A=W\Theta V^\T$ be the $S^\T S$-SVD of $A$. Then
\begin{equation}\label{eqn:P}
P=WV^\T,
\end{equation}
is still $S^\T S$-orthogonal thanks to Remark~\ref{rem:invariance}.
Moreover, $P$ is the orthogonal factor of the \emph{randomized} polar decomposition of $A$. Indeed, we can write
\begin{equation}\label{eq:rand_polar}
    A=PH,\quad {\rm with}\,\,
H=V\Theta V^\T,
\end{equation}
with $\Theta={\rm diag}(\theta_1, \ldots, \theta_n)$ and
$\theta_i\ge 0$, so that $H$ is positive semidefinite. {\color{black} Notice that, to the best of our knowledge, a randomization-based polar decomposition has never been proposed prior this work. Indeed, the decomposition~\eqref{eq:rand_polar} is a natural consequence of our $S^\T S$-SVD thanks to the availability of the left singular matrix $W$.}

In the next theorem we prove that $P$ defined  in (\ref{eqn:P}) solves
the problem (\ref{eq:main}).
To this end, we notice that we
 can parametrize all $S^\T S$-orthogonal matrices spanning Range$(A)$
 by introducing the following set
$$
{\cal Q}_S(A)= \{ Q\in{\mathbb R}^{m\times n}
\, : \, Q=W L V^\T, \,\, {\rm with}\,\, L^\T L=I, \,A=W\Theta V^\T\}.
$$ 

The following result holds.
%Then the following result holds.

\begin{theorem}
Let $A=W\Theta V^\T$ be the $S^\T S$-SVD of $A$ where $S$ is a {\color{black}randomized $(\varepsilon,\delta,\text{rank}(A))$}-subspace
embedding of $\text{Range}(A)$.
 Let $A=PH$ with $P\in {\cal Q}_S(A)$ defined in (\ref{eqn:P}) and
$H=V\Theta V^\T$. Then with high probability,
 \begin{equation}
 \|A-P\|_{S,*}=
\min_{Q\in {\cal Q}_S(A)}
%\substack{Q\in\mathbb{R}^{m\times n},
% \\ Q^\T S^\T S=I}}
\|A-Q\|_{S,*},
\end{equation}
where $*=2,F$.
\end{theorem}

\begin{proof}
 Using $A=PH$ and the $S^\T S$-orthogonality of $P$, we first notice that
 $$
\|A-P\|^2_{S,*}=\|P(H-I)\|^2_{S,*}=\|H-I\|_*^2=
\|\Theta-I\|_*^2.
 $$
%
% $$
%\|A-P\|^2_{S^\T S}=\|WU(\Sigma-I)V^\T\|_{S^T S}^2=\text{trace}(V(\Sigma-I)U^TW^T S^T S(\Sigma-I)V^T)=\|\Sigma-I\|_F^2.$$
% %%
% Moreover, any $Q\in\mathbb{R}^{m\times n}$ such that $Q^T S^T S=I$ can be written as
 %%
% $$Q=WULV^T,$$
%%
%where the matrix $L\in\mathbb{R}^{n\times n}$ takes into account the coefficients of the change of basis and it must be orthogonal to preserve the $S^T S$-orthogonality of $Q$.

We first focus on $*=F$. For the same reasonings as above, for any $Q\in {\cal Q}_S(A)$ and using
that $L^\T L=I$, we have
\begin{align*}
 \SF{A-Q}^2=&\,\|\Theta-L\|_F^2=\|\Theta -I+I-L\|_F^2\\
 =&\,{\rm trace}\left((\Theta-I)^2+(\Theta-I)(I-L)+(I-L)^\T(\Theta-I)+(I-L)^\T(I-L)\right)\\
 =&\,{\rm trace}\left((\Theta-I)^2+\Theta(I-L)+(I-L)^\T\Theta\right)
 =\,\|\Theta-I\|_F^2+  2\sum_{i=1}^n\theta_i(1-\ell_{ii}).
\end{align*}
%thanks to the orthogonality of $L$, i.e., $L^\TL=I$.
%
The orthogonality of $L$ also implies that $\ell_{i,i}\leq 1$ with 
the equality attained for all $i$s if and only if $L=I$, i.e., $Q=P$. Therefore, 
 $2\sum_{i=1}^n\theta_i(1-\ell_{ii})\geq0$ for $L\ne I$.
%$\text{trace}\left(\Sigma(I-L)\right)$, 
%$\text{trace}\left((I-L)^\T\Sigma\right)\geq 0$.
In conclusion, we thus have
$$\SF{A-Q}^2\geq \|\Theta-I\|_F^2=\SF{A-P}^2,$$
which shows the first result.

For $*=2$, we have that
$$\Sdue{A-P}^2=\max_{\|x\|_2=1} x^\T(\Theta-I)^2x=\max_{i=1,\ldots,n}(\theta_i-1)^2,
$$
which means that the vector $x$ attaining the maximum is one of the vectors of the canonical basis of $\mathbb{R}^n$, namely an $e_{\bar p}$ for a certain $\bar p\in\{1,\ldots,n\}$.

For a generic $S^\T S$-orthogonal matrix $Q=WLV^T$ it holds

\begin{align*}
 \Sdue{A-Q}^2=&\;\|\Theta-L\|_2^2=\| \Theta-I+I-L\|_2^2=\max_{\|x\|_2=1} x^\T\left((\Theta-I)^2+\Theta(I-L)+(I-L)^\T\Theta\right)x\\
 \geq&\;
e_{\bar p}^\T\left((\Theta-I)^2+\Theta(I-L)+(I-L)^\T\Theta\right)e_{\bar p}\\
=&\;\Sdue{A-P}^2+e_{\bar p}^\T\left(\Theta(I-L)+(I-L)^\T\Theta\right)e_{\bar p}.
\end{align*}
For the same reasonings as above, 
$e_{\bar p}^\T\left(\Theta(I-L)+(I-L)^\T\Theta\right)e_{\bar p}\geq 0$, and the 
result {\color{black}follows. $\square$}
\end{proof}

Generalizations of the polar decomposition when using {\it deterministic and positive definite} inner products
have been discussed in detail in \cite{doi:10.1137/090765018}.

The columns of the matrix $P$ are $S^\T S$-orthogonal but not orthogonal. Hence, a natural question is how
close $P$ is to a matrix with orthonormal columns.
To this end, let $P=Q_P H_P$ be the polar decomposition of $P$,
where $Q_P$ has orthonormal columns and is the closest such
matrix to $P$.
The following bound was shown in \cite[Lemma 5.1]{Higham.94} for a general $P$
of full column rank,
\begin{equation}\label{Bound:Higham}
\frac{\|P^\T P-I\|_2}{\|P\|_2+1} \le
\|P-Q_P\|_2
\le
\|P^\T P-I\|_2.
\end{equation}
For $P$ being $S^\T S$-orthogonal, by using Proposition~\ref{prop:orthP} 
we thus have, with high probability,

\begin{equation}\label{eqn:PQnorm}
\|P-Q_P\|_2 
\le
\|P^\T P-I\|_2\le
\frac{\varepsilon}{1-\varepsilon} .
\end{equation}

Analogously, given a matrix $T$ with orthonormal columns,
we can estimate how
close $T$ is to a matrix with $S^\T S$-orthonormal columns.

\begin{lemma}
Given a matrix $T$ with orthonormal columns,
let $Q_T$ be the $S^\T S$-orthogonal factor of the $S^\T S$-polar decomposition
of the matrix $T$. Then with high probability  
%\todo{VS:to be checked DP: checked. I had to add one step because it was not clear to me at first.}
\begin{equation}\label{eqn:TQnorm}
\Sdue{T-Q_T}
\le
{\varepsilon} .
%\frac{\varepsilon}{1-\varepsilon} .
\end{equation}
\end{lemma}

\begin{proof}
We need to generalize the upper bound in (\ref{Bound:Higham}). To this end,
we follow the proof in \cite[Lemma 5.1]{Higham.94}.
Let $T=Q_T H$ be the $S^\T S$-polar decomposition of $T$, with $H$ symmetric
and positive definite.
In particular, $(ST)^\T ST=H^2$. 
%Let $\Theta$ be the matrix of eigenvalues of $H$.
With explicit computation we can verify that 
\begin{equation}\label{eqn:T S}
T^\T S^\T S T-I =
(T-Q_T)^\T S^\T S(T+Q_T)=
(T-Q_T)^\T S^\T S Q_T(H+I).
\end{equation}
Since $H+I$ is nonsingular, we can write
$$(T-Q_T)^\T S^\T S Q_T= (T^\T S^\T S T-I)(H+I)^{-1},$$
so that
$$
\|(T-Q_T)^\T S^\T S Q_T\|_2\le 
\|T^\T S^\T S T-I\|_2 \|(H+I)^{-1}\|_2\le 
\|T^\T S^\T S T-I\|_2.
$$
Moreover, 
$$
\Sdue{T-Q_T} = \|H-I\|_2=
 \|Q_T^\T S^\T S Q_T(H-I) \|_2=
 \|Q_T^\T S^\T S (T-Q_T) \|_2,
% \|(T-Q_T)^\T S^\T S Q_T\|_2,
$$
so that
$$
\Sdue{T-Q_T}\leq \|T^\T S^\T S T-I\|_2,
$$
and the result follows from Proposition~\ref{prop:orthT}. \hfill {\color{black} $\square$}
\end{proof}

We are now left to evaluate the quality of the minimizer when either orthogonality or sketched orthogonality is used.

\begin{proposition} \label{Prop_relate_nearestproblems}
 Let $P$ solve (\ref{eq:main}), and let
$T$ solve (\ref{eq:main0}) in the 2-norm.
With the previous notation, with high probability it holds 
$$
\|A-T\|_2-\frac{\varepsilon}{1-\varepsilon} 
\le
\|A-P\|_2 
\le
\frac{1+\varepsilon}{1-\varepsilon}
\|A-T\|_2+\frac{\varepsilon}{1-\varepsilon} .
$$
\end{proposition}

\begin{proof}
Let $Q_P$ be the orthogonal factor of the polar decomposition of $P$.
Then
$$
\|A-T\|_2\le \|A-Q_P\|_2\le \|A-P\|_2+\|P-Q_P\|_2 
\le \|A-P\|_2+\frac{\varepsilon}{1-\varepsilon} ,
$$
where, in the last inequality, (\ref{eqn:PQnorm}) was used. 
This gives the left-hand side bound.
To obtain the bound on the right side, let $Q_T$ be the $S^\T S$-orthogonal factor
of the $S^\T S$-polar decomposition of $T$. Then
$$
\Sdue{A-P}\le
\Sdue{A-Q_T}\le
%(1+\varepsilon)\|A-Q_T\|_2\le (1+\varepsilon)
\Sdue{A-T} + \Sdue{T-Q_T}
\le
(1+\varepsilon)\|A-T\|_2 + \varepsilon
$$
where in the last inequality, (\ref{eqn:TQnorm}) was used.

The result follows from recalling that
$(1-\varepsilon)\|A-P\|_2\le \Sdue{A-P}$. \hfill {\color{black} $\square$}
\end{proof}

Proposition~\ref{Prop_relate_nearestproblems} shows that when solving~\eqref{eq:main}, the minimizer $P$ attains an error in the 2-norm which is not too far from the best attainable error (in that norm). Moreover, the cheaper computation of the $S^TS$-SVD could also make the exact solution of~\eqref{eq:main} affordable in the case of large dimensional problems.
Notice that from the proof of Proposition~\ref{Prop_relate_nearestproblems}, it is evident that the factor
$\frac{1+\varepsilon}{1-\varepsilon}$ arises in relating the
matrix 2-norm and the $S,2$-norm.

\begin{example}\label{ex:ex2}
{\rm
We consider the matrix $A\in\mathbb{R}^{m\times n}$ stemming from the benchmark problem {\tt abtaha2} in the SuiteSparse Matrix Collection Repository\footnote{Available at
{\tt https://sparse.tamu.edu/}.}. The matrix has size
$37932\times 331$. We compute the nearest orthogonal matrix $T$ to $A$ in the 2-norm, that is, we solve~\eqref{eq:main0} for $*=2$, and 
compare the obtained solution with $P$ coming from  solving~\eqref{eq:main} for the $S,2$ norm. To this end, we consider the sketching~\eqref{eq:sketching_trig} for different values of $s$.

To construct $P$, we first compute the factors of the $S^\T S$-SVD of $A$, i.e., $A=W\Theta V^\T$ as shown in section~\ref{SS-SVD}. We then set $P=WV^\T$. Similarly, to obtain $T$, we compute the 
SVD of $A$, $A=U\Sigma Y^\T$, by the matlab function {\tt svd}, and set $T=UY^\T$.
This choice of $T$ yields $\|A-T\|_2=24.77$,  and the time devoted 
to its computation is 1.09 seconds.  

In Table~\ref{tab:ex2} we report the distance from $A$ attained 
by $P$, measured in the 2-norm. This is  not far from the one 
provided by $T$, as predicted by Proposition~\ref{Prop_relate_nearestproblems}. 
Indeed, for $\varepsilon=0.5$, 
Proposition~\ref{Prop_relate_nearestproblems} states that, with high probability,
$$\|A-T\|_2-1\leq\|A-P\|_2\leq 3\|A-T\|_2+1,$$
and these bounds are satisfied  for all tested values of $s$.
We also report $\|P-T\|_2$ and the running times for computing $P$ while varying $s$. 
Looking at the results in Table~\ref{tab:ex2}, we can see that
$\|P-T\|_2$ is always rather moderate, with up to
55\% cuts in running time when computing $P$ instead of $T$.

  \begin{table}[t!]
   \centering
   \begin{tabular}{rrrr}
     &    $S^\T S$-orth  & Distance & \\
 $s$ & $\|A-P\|_2$ &  $\|P-T\|_2$ & Time (s)\\
    \hline
    $2n$ & 24.99& 3.69 & 0.46\\
    $4n$ & 24.80& 2.57& 0.46\\
    $6n$ & 24.77& 2.36& 0.47\\
    $8n$ & 24.76& 2.27& 0.48\\
    $10n$ & 24.75 & 2.21& 0.48\\
    $12n$ & 24.74& 2.18& 0.49\\
   \end{tabular}
\caption{Example~\ref{ex:ex2}. 2-norm distance with the solution $P$ to~\eqref{eq:main}, 2-norm difference between $P$ and the solution $T$ to~\eqref{eq:main0}, and the run times for computing $P$, as the sketching dimension $s$ varies.
The reported results are averaged over 50 runs. 
\label{tab:ex2}}
  \end{table}
  }
\end{example}

%%%%%%%%%%%%%%%%%%%%%%%%%%%%%%%%%%%%%
\section{Conclusions}\label{Conclusions}
We have formally introduced the randomization-based
 $S^\T S$-SVD decomposition of a given tall matrix $A$, which holds with high
probability. 
This decomposition resembles the standard SVD, where, however,
the left singular vectors are constrained to be orthonormal with respect to the $S^\T S$-norm. 
The $S^\T S$-SVD has then been employed to derive a number of results scattered in the literature, and to directly derive probabilistic bounds on the distance from (standard) orthogonality of the sketched orthogonal factor,
also in state-of-the-art randomized {\color{black}algorithms for the QR factorization}.
We believe that the $S^\T S$-SVD has the potential to fully characterize the behavior of sketching techniques applied to least squares problems.

We have also studied the related problem of finding the nearest orthogonal matrix to $A$, in the same $S^\T S$-norm,
 and adopted the $S^\T S$-SVD of $A$ for its solution. Additional comparison
bounds have complemented our presentation, illustrating that sketched orthogonality
allows one to obtain results comparable to those with Euclidean orthogonality, at a
significant lower computational cost.

%%%%%%%%%%%%%%%%%%%%%%%%%%%%%%%%%%%%%
\section*{Acknowledgments}
Both authors are members of the INdAM Research
Group GNCS. Moreover, their work
was partially supported by the European Union - NextGenerationEU under the National Recovery and Resilience Plan (PNRR) - Mission 4 Education and research
- Component 2 From research to business - Investment 1.1 Notice Prin 2022 - DD N. 104 of 2/2/2022,
entitled “Low-rank Structures and Numerical Methods in Matrix and Tensor Computations and their
Application”, code 20227PCCKZ – CUP J53D23003620006.

\section*{Conflicts of interest}
Not applicable

\bibliographystyle{siamplain}
\bibliography{references}

@article{Bjorck,
author = {Bj\"orck, \AA{A}. and Bowie, C.},
title = {An Iterative Algorithm for Computing the Best Estimate of an Orthogonal Matrix},
journal = {SIAM Journal on Numerical Analysis},
volume = {8},
number = {2},
pages = {358-364},
year = {1971},
doi = {10.1137/0708036},
}

@INPROCEEDINGS{Sarlos2006,
  author={Sarlos, T.},
  booktitle={2006 47th Annual IEEE Symposium on Foundations of Computer Science (FOCS'06)}, 
  title={{Improved Approximation Algorithms for Large Matrices via Random Projections}}, 
  year={2006},
  volume={},
  number={},
  pages={143-152},
  doi={10.1109/FOCS.2006.37}}

@article{Blendenpik,
author = {Avron, H. and Maymounkov, P. and Toledo, S.},
title = {{Blendenpik: Supercharging LAPACK's Least-Squares Solver}},
journal = {SIAM Journal on Scientific Computing},
volume = {32},
number = {3},
pages = {1217-1236},
year = {2010},
doi = {10.1137/090767911},
}

@article{RokTyg2008,
author = {V. Rokhlin  and M. Tygert },
title = {A fast randomized algorithm for overdetermined linear least-squares regression},
journal = {Proceedings of the National Academy of Sciences},
volume = {105},
number = {36},
pages = {13212-13217},
year = {2008},
doi = {10.1073/pnas.0804869105},}

@article{BecT17,
author = {Beckermann, B. and Townsend, A.},
title = {{On the Singular Values of Matrices with Displacement Structure}},
journal = {SIAM Journal on Matrix Analysis and Applications},
volume = {38},
number = {4},
pages = {1227-1248},
year = {2017},
doi = {10.1137/16M1096426},
}

@article{WOOLFEetal2008,
title = {A fast randomized algorithm for the approximation of matrices},
journal = {Applied and Computational Harmonic Analysis},
volume = {25},
number = {3},
pages = {335-366},
year = {2008},
doi = {10.1016/j.acha.2007.12.002},
author = {Franco Woolfe and Edo Liberty and Vladimir Rokhlin and Mark Tygert},
}

@article{Martinsson_Tropp_2020, 
title={{Randomized numerical linear algebra: Foundations and algorithms}}, 
volume={29}, 
DOI={10.1017/S0962492920000021}, 
journal={Acta Numerica}, 
author={Martinsson, P.-G. and Tropp, J. A.}, 
year={2020}, 
pages={403–572}}

@InProceedings{Gilbertetal.14,
  title={{Sketched SVD: Recovering Spectral Features from Compressive Measurements}},
  author={A. C. Gilbert and J. Y. Park and M. B. Wakin},
Booktitle={Proc. of SPARS 2013}, year=2013,
  note={(also ArXiv: 1211.0361, 2012)},
}

@article{RandGSreorth,
author = {Jang, Y. and Grigori, L.},
title = {Randomized Orthogonalization Process With Reorthogonalization},
journal = {Numerical Linear Algebra with Applications},
volume = {32},
number = {4},
pages = {e70029},
doi = {10.1002/nla.70029},
year = {2025}
}

@techreport{RandrankrevealingQR,
  AUTHOR = {Grigori, L. and Xue, Z.},
  TITLE = {{Randomized strong rank-revealing QR for column subset
selection and low-rank matrix approximation}},
Institution={ArXiv},
year = {2025},
  number={2503.18496}
}

@article{RandomGS,
author = {Balabanov, O. and Grigori, L.},
title = {{Randomized Gram--Schmidt Process with Application to GMRES}},
journal = {SIAM Journal on Scientific Computing},
volume = {44},
number = {3},
pages = {A1450-A1474},
year = {2022},
doi = {10.1137/20M138870X},
}

@article{BalabanovNouy19,
	author = {O. Balabanov and A. Nouy},
	title = {Randomized linear algebra for model reduction. {Part I}: Galerkin methods and error estimation},
	journal = {Adv Comput Math},
	volume = {45},
	pages = {2969–3019},
	year = {2019},
	doi = {10.1007/s10444-019-09725-6},
}

@article{RSVD_singularvalues2,
author = {Saibaba, A. K.},
title = {{Randomized Subspace Iteration: Analysis of Canonical Angles and Unitarily Invariant Norms}},
journal = {SIAM Journal on Matrix Analysis and Applications},
volume = {40},
number = {1},
pages = {23-48},
year = {2019},
doi = {10.1137/18M1179432},
}

@article{VanLoan76,
 author = {Van~Loan, C. F.},
 journal = {SIAM Journal on Numerical Analysis},
 number = {1},
 pages = {76--83},
 title = {{Generalizing the Singular Value Decomposition}},
 volume = {13},
 year = {1976}
}

@article{Halko2010,
author = {Halko, N. and Martinsson, P. G. and Tropp, J. A.},
title = {{Finding Structure with Randomness: Probabilistic Algorithms for Constructing Approximate Matrix Decompositions}},
journal = {SIAM Review},
volume = {53},
number = {2},
pages = {217-288},
year = {2011},
doi = {10.1137/090771806},
}

@techreport{RandGSHouse,
author = {Grigori, L. and Timsit, E.},
title = {{Randomized Householder QR}},
year = {2024},
institution = {ArXiv},
number = {2405.10923}
}

@Article{Szyldetal2024,
author = {Higgins, A. J. and Szyld, D. B. and Boman, E. G. and Yamazaki, I.},
title = {{Analysis of Randomized Householder--Cholesky QR Factorization with Multisketching}},
year = {2025},
  journal    = {Numer. Math.},
  volume     = {157},
  pages      = {1695--1737},
  doi        = {10.1007/00211-025-01492-5}
  }

@Book{Golub.VLoan.13,
  Title                    = {Matrix Computations},
  Author                   = {G. Golub and {C. F. ~Van~Loan}},
  Publisher                = {The Johns Hopkins University Press},
  Year                     = {2013},
  
  Address                  = {Baltimore},
  Edition                  = {4th},

  Mynote                   = {it was GoVa89}
}

@article{Tropp2011,
author = {Tropp, J. A.},
title = {{Improved Analysis of the Subsampled Randomized Hadamard Transform}},
journal = {Advances in Adaptive Data Analysis},
volume = {03},
number = {01n02},
pages = {115-126},
year = {2011},
doi = {10.1142/S1793536911000787},
}

@article{burke2023gmres,
      author = {Burke, L. and G\"{u}ttel, S. and Soodhalter, K. M.},
title = {{GMRES} with Randomized Sketching and Deflated Restarting},
journal = {SIAM Journal on Matrix Analysis and Applications},
volume = {46},
number = {1},
pages = {702-725},
year = {2025},
doi = {10.1137/23M1619472}}

@article{Higham.94,
title={The Matrix Sign Decomposition
and Its Relation to the Polar Decomposition},
author={Nicholas J. Higham},
journal={Linear Algebra and its Applications},
volume={212/213}, pages={3--20}, year=1994
}

@article{doi:10.1137/090765018,
author = {Higham, Nicholas J. and Mehl, Christian and Tisseur, Fran\c{c}oise},
title = {The Canonical Generalized Polar Decomposition},
journal = {SIAM Journal on Matrix Analysis and Applications},
volume = {31},
number = {4},
pages = {2163-2180},
year = {2010},
URL = { https://doi.org/10.1137/090765018 }
}

@Article{eas:99,
  Title                    = {The geometry of algorithms with orthogonality constraints},
  Author                   = {A. Edelman and T. Arias and S.~T. Smith},
  Journal                  = {SIAM J. Matrix Anal. Appl.},
  Year                     = {1999},
  Number                   = {2},
  Pages                    = {303 --353},
  Volume                   = {20}
}

@Book{Absiletal.08,
  author    = {P.-A. Absil and R. Mahony and R. Sepulchre},
  publisher = {Princeton University Press},
  title     = {Optimization Algorithms on Matrix Manifolds},
  year      = {2008},
  address   = {Princeton, NJ},
}

@Book{Hairer2002,
  Title                    = {Geometric numerical integration. Structure-preserving algorithms for ordinary differential equations},
  Author                   = {E. Hairer and Ch. Lubich and G. Wanner},
  Publisher                = {Springer},
  Year                     = {2002},

  Address                  = {Berlin},
  Series                   = {Springer Series in Computational Mathematics},
  Volume                   = {31}
}

@Article{Celledoni2002,
  Title                    = {A class of intrinsic schemes for orthogonal integration},
  Author                   = {E. Celledoni and B. Owren},
  Journal                  = {SIAM J. Numer. Anal.},
  Year                     = {2002},
  Number                   = {6},
  Pages                    = {2069-2084},
  Volume                   = {40}
}

@article{Cerutietal.22,
author={G. Ceruti and J. Kusch and C. Lubich}, 
title={A rank-adaptive robust integrator for dynamical lowrank approximation}, 
journal={BIT Numer. Math.}, volume=62, year=2022, pages={1149--1174}
}

@Article{VarahApril1973a,
  Title                    = {On the numerical solution of ill-conditioned linear systems with applications to ill-posed problems},
  Author                   = {J. M. Varah},
  Journal                  = {SIAM J. Numer. Anal.},
  Year                     = {April 1973},
  Number                   = {2},
  Pages                    = {257-267},
  Volume                   = {10}
}

@article{doi:10.1137/140977898,
author = {Boutsidis, Christos and Woodruff, David P.},
title = {Optimal CUR Matrix Decompositions},
journal = {SIAM Journal on Computing},
volume = {46},
number = {2},
pages = {543-589},
year = {2017},
doi = {10.1137/140977898}
}

@Book{Hansen.book.10,      
  Title                    = {Discrete inverse problems: insight and algorithms},
  Author                   = {Per~Christian Hansen},
  Publisher                = {SIAM},
  Year                     = {2010}
}

@book{Bjorck.book.96,
author = {Bj{\"o}rck, Ake},
title = {Numerical Methods for Least Squares Problems},
publisher = {Society for Industrial and Applied Mathematics},
year = {1996},
doi = {10.1137/1.9781611971484},
address = {},
edition   = {}
}

@techreport{Kireeva.Tropp.tr23,
  doi = {10.7907/7YADE-5K351},
  url = {https://authors.library.caltech.edu/records/4v7j7-msj10},
  author = {Kireeva, Anastasia and Tropp, Joel A.},
  institution={arXiv:2402.17873, v.3},
  title = {Randomized matrix computations: themes and variations},
  year = {2025}
}

\end{document}